\documentclass[10pt,a4paper]{article}
\usepackage[utf8]{inputenc}
\usepackage[english]{babel}
\usepackage{hyperref}
\usepackage{amsmath,amsthm,enumerate}
\usepackage{amssymb}
\usepackage{graphicx}
\usepackage{tikz}
\usepackage{ulem,xpatch}
\usepackage[hmargin=2.5cm,vmargin=3cm]{geometry}

\newtheorem{theorem}{Theorem}
\newtheorem{proposition}[theorem]{Proposition}
\newtheorem{lemma}[theorem]{Lemma}
\newtheorem{observation}[theorem]{Observation}

\newtheorem{definition}[theorem]{Definition}

\newtheorem*{question}{Question}
\newtheorem{conjecture}[theorem]{Conjecture}

\newcommand{\smallqed}{\hfill{\tiny $\Box$}}
\newcommand{\F}{\mathcal{C}_5^+}
\DeclareMathOperator{\free}{free}
\newenvironment{claim}{\mbox{}\\ \noindent \it}{\\[1ex]}
\newenvironment{claim_proof}{\noindent}{\smallqed\\}

\usetikzlibrary{decorations.pathreplacing}
\usetikzlibrary{decorations.pathmorphing}
\tikzstyle{jolinoeud}=[circle,inner sep=2, minimum size =3 pt, line width = 1pt, draw=black, fill=white]

\usepackage[color=green!30]{todonotes}

\begin{document}

\title{Strengthening the Murty-Simon conjecture\\on diameter $2$ critical graphs}
\author{Antoine Dailly\footnote{\noindent Univ Lyon, Universit\'e Claude Bernard, CNRS, LIRIS - UMR 5205, F69622, France.} \footnote{Univ. Grenoble Alpes, CNRS, Grenoble INP, G-SCOP, Grenoble, France. E-mail: antoine.dailly@grenoble-inp.fr}
\and Florent Foucaud\footnote{\noindent LIMOS - CNRS UMR 6158, Universit\'e Clermont Auvergne, Aubi\`ere, France. E-mail: florent.foucaud@gmail.com}
\and Adriana Hansberg\footnote{\noindent Instituto de Matem\'aticas, UNAM Juriquilla, 76230 Quer\'etaro, Mexico. E-mail: ahansberg@im.unam.mx}  \footnote{\noindent Partially supported by PAPIIT  IA103915, PAPIIT IA103217, and CONACyT project 219775.}
\and 
}

\maketitle

\begin{abstract}
  A graph is \emph{diameter-$2$-critical} if its diameter is~$2$ but the removal of any edge increases the diameter. A well-studied conjecture, known as the Murty-Simon conjecture, states that any diameter-$2$-critical graph of order $n$ has at most $\lfloor n^2/4\rfloor$ edges, with equality if and only if $G$ is a balanced complete bipartite graph. Many partial results about this conjecture have been obtained, in particular it is known to hold for all sufficiently large graphs, for all triangle-free graphs, and for all graphs with a dominating edge. In this paper, we discuss ways in which this conjecture can be strengthened. Extending previous conjectures in this direction, we conjecture that, when we exclude the class of complete bipartite graphs and one particular graph, the maximum number of edges of a diameter-$2$-critical graph is at most $\lfloor(n-1)^2/4\rfloor+1$. The family of extremal examples is conjectured to consist of certain twin-expansions of the $5$-cycle (with the exception of a set of thirteen special small graphs). Our main result is a step towards our conjecture: we show that the Murty-Simon bound is not tight for non-bipartite diameter-$2$-critical graphs that have a dominating edge, as they have at most $\lfloor n^2/4\rfloor-2$ edges. Along the way, we give a shorter proof of the Murty-Simon conjecture for this class of graphs, and stronger bounds for more specific cases. We also characterize diameter-$2$-critical graphs of order $n$ with maximum degree $n-2$: they form an interesting family of graphs with a dominating edge and $2n-4$ edges.
\end{abstract}

\section{Introduction}

A graph $G$ is called \emph{diameter-$d$-critical} if its diameter is $d$, and the deletion of any edge increases the diameter. Diameter-$d$-critical graphs and their extremal properties have been studied since at least the 1960s, see for example the few selected references~\cite{bollobas,CH79,erdos,Loh-Ma,ore,plesnik,plesnik2}.

The case of diameter-$2$-critical graphs (D2C graphs for short), being the simplest nontrivial case, has been the focus of much work. Many famous and beautiful graphs are D2C. Such examples are: complete bipartite graphs; the $5$-cycle (more generally, the set of Moore graphs of diameter $2$, the others being the Petersen graph, the Hoffman-Singleton graph, and a hypothetical graph of order $3250$~\cite{HoSi60}); the Wagner graph; the Chv\'atal graph; the Gr\"otzsch graph; the Clebsch graph... When the order is at least~$3$, D2C graphs coincide with \emph{minimally triangle-saturated graphs}, see~\cite{MacDougall-Eggleton}.

In the 1960s and 1970s, Ore, Murty, Ple\'snik and Simon independently stated the following conjecture, generally known under the name of \emph{Murty-Simon conjecture} (see~\cite{CH79,ore,plesnik}). A complete bipartite graph is \emph{balanced} if the sizes of the two parts differ by at most $1$.

\begin{conjecture}\label{conj:MS}
Any D2C graph of order $n$ has at most $\left\lfloor n^2/4\right\rfloor$ edges, with equality if and only if $G$ is a balanced complete bipartite graph.
\end{conjecture}

Many partial results towards Conjecture~\ref{conj:MS} have been obtained. The best general upper bound is due to Fan~\cite{Fan}, who showed that any D2C graph of order $n$ has less than $0.2532n^2$ edges. In the same paper, Fan also proved the conjecture for graphs of order at most~$24$ and~$26$. On the other hand, using the Regularity Lemma, F\"uredi~\cite{Furedi} proved Conjecture~\ref{conj:MS} for graphs with order $n>n_0$, where $n_0$ is a gigantic number: roughly a tower of $2$'s of height $10^{14}$.

It is not difficult to observe that any bipartite graph of diameter~$2$ must be a complete bipartite graph. Thus, Conjecture~\ref{conj:MS} restricted to bipartite graphs is well-understood; when studying D2C graphs, we can consider only non-bipartite graphs.

Hanson and Wang observed in~\cite{HW03} that a non-bipartite graph is D2C if and only if its complement is $3$-total domination critical. This launched a fruitful research path that led to many results around Conjecture~\ref{conj:MS}: see the papers~\cite{BHHH,HH12,diam3,mindeg,correction,clawfree,connect,diam2or3} and the survey~\cite{survey}.

Conjecture~\ref{conj:MS} holds for triangle-free graphs: this is exactly Mantel's theorem~\cite{mantel}, which states that a triangle-free graph of order $n$ has at most $\lfloor n^2/4\rfloor$ edges, with equality if and only if $G$ is a balanced complete bipartite graph. It is not difficult to observe that a graph is both D2C and triangle-free if and only if it is \emph{maximal triangle-free}~\cite{Barefoot} (and equivalently, triangle-free with diameter $2$). Maximal triangle-free graphs are widely studied, see for example~\cite{BS14,MDE02,GK93,MacDougall-Eggleton}.

It is known that no triangle-free non-bipartite D2C graph has a dominating edge~\cite{BHHH} (a \emph{dominating edge} is a pair of adjacent vertices that have no common non-neighbour). Thus, another case of interest is the set of D2C graphs \emph{with} a dominating edge.\footnote{Equivalently, the set of D2C graphs with total domination number $2$.} A non-bipartite D2C graph has a dominating edge if and only if its complement has diameter~$3$~\cite{HW03} (note that the complement of a non-bipartite D2C graph has diameter either $2$ or $3$~\cite{diam2or3}). Conjecture~\ref{conj:MS} was proved for D2C graphs with a dominating edge in the series of papers~\cite{HW03,diam3,correction,Wang-arxiv}.

It has been observed that the bound of Conjecture~\ref{conj:MS} might be strengthened. Such a claim was made by F\"uredi in the article~\cite{Furedi} containing his proof of the conjecture for very large graphs. In the conclusion of~\cite{Furedi} (Theorem~7.1), he claimed, without proof, that his method can be extended to show the following: any sufficiently large non-bipartite D2C graph of order $n$ has at most $\lfloor(n-1)^2/4\rfloor+1$ edges, with equality if and only if the graph is obtained from an even-order balanced complete bipartite graph by subdividing an edge once. Such a graph is also obtained from a $5$-cycle by replacing two adjacent vertices by two same-size independent sets of \emph{twins} (vertices with the same open neighbourhood). Nevertheless, one might perhaps need to consider F\"uredi's claim with caution, as it appears to be partly false. Indeed, it was observed in~\cite{BHHH} that one can also obtain a D2C graph with the same number of edges by replacing \emph{three} vertices $x_1$, $x_2$, $x_3$ of a $5$-cycle by three independent sets $X_1$, $X_2$, $X_3$ of twins, under the following conditions: (1) $x_1$, $x_2$ and $x_3$ are consecutive on the $5$-cycle; (2) $|X_2|\in\{\lfloor\tfrac{n-3}{2}\rfloor,\lceil\tfrac{n-3}{2}\rceil\}$, where $n$ is the number of vertices of the obtained graph. Let us call $\F$, the family of these ``expanded $5$-cycles'' that satisfy (1) and (2): the construction is depicted in Figure~\ref{fig:C5-expansion}. 

\begin{figure}[!ht]
  \centering
\begin{tikzpicture}[scale=1]
  \node[jolinoeud,scale=0.8](0) at (360/5*3+18:1) {};
  \node[jolinoeud,scale=0.8](1) at (360/5*4+18:1) {};
  \draw (0)--(1);
  \draw (360/5*2+18:1) circle (0.4);
  \draw (360/5*1+18:1) circle (0.4);
  \draw (360/5*0+18:1) circle (0.4);
  \draw (360/5*1+18:1) node[scale=0.6] {$X_2$};
  \draw (360/5*0+18:1) node[scale=0.6] {$X_1$};
  \draw (360/5*2+18:1) node[scale=0.6] {$X_3$};
  \tikzset{shift={(360/5*2+18:1)}}
  \draw (0)--(220:0.4);
  \draw (0)--(360:0.4);
  \node (b1) at (120:0.4) {};
  \node (b2) at (330:0.4) {};
  \tikzset{shift={(360/5*2+18:-1)}}
  \tikzset{shift={(360/5*0+18:1)}}
  \draw (1)--(320:0.4);
  \draw (1)--(180:0.4);
  \node (c1) at (60:0.4) {};
  \node (c2) at (210:0.4) {};
  \tikzset{shift={(360/5*0+18:-1)}}
  \tikzset{shift={(360/5*1+18:1)}}
  \node (a1) at (120:0.4) {};
  \node (a3) at (60:0.4) {};
  \node (a2) at (270:0.4) {};
  \node (a4) at (270:0.4) {};
  \tikzset{shift={(360/5*1+18:-1)}}
  \draw (a1.center)--(b1.center);
  \draw (a2.center)--(b2.center);
  \draw (a3.center)--(c1.center);
  \draw (a4.center)--(c2.center);
\end{tikzpicture}
\caption{The infinite family $\F$ of expanded $5$-cycles: we have $|X_2|\in\{\lfloor\tfrac{n-3}{2}\rfloor,\lceil\tfrac{n-3}{2}\rceil\}$.}
\label{fig:C5-expansion}
\end{figure}
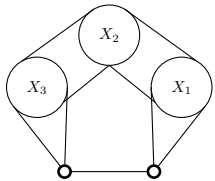

The authors of~\cite{BHHH}, unaware of F\"uredi's claim, focused their attention on the case of non-bipartite D2C graphs without a dominating edge (recall that this class includes all non-bipartite D2C triangle-free graphs).\footnote{Note that the terminology used in~\cite{BHHH} is the equivalent one of total domination criticality for the complement of the graphs.} They conjectured that, for this class, the corrected version of F\"uredi's claim is true:

\begin{conjecture}[Balbuena, Hansberg, Haynes, Henning \cite{BHHH}]\label{conj:BHHH}
Let $G$ be a non-bipartite D2C graph without a dominating edge. Then, $G$ has at most $\left\lfloor(n-1)^2/4\right\rfloor+1$ edges. For sufficiently large $n$, equality holds if and only if $G$ belongs to $\F$.
\end{conjecture}

Conjecture~\ref{conj:BHHH} was proved in~\cite{BHHH} for triangle-free graphs without the restriction on the order (the first part of the statement was also proved independently for triangle-free graphs in~\cite{Barefoot}).

One non-bipartite D2C graph that does not satisfy the bound $\lfloor(n-1)^2/4\rfloor+1$ is known: it is the graph $H_5$ of Figure~\ref{fig-h5} (which has six vertices, eight edges, and a dominating edge). This fact was observed in~\cite{Barefoot,HH12}; the authors of~\cite{Barefoot} asked whether this graph is the only exception to the bound.

\begin{figure}[!ht]
	\centering
	\begin{tikzpicture}
		\node (h5) at (0,0) {
			\begin{tikzpicture}
			\node[jolinoeud](0) at (0,0.5) {};
			\node[jolinoeud](1) at (0.5,0) {};
			\node[jolinoeud](2) at (0.5,1) {};
			\node[jolinoeud](3) at (1.5,0) {};
			\node[jolinoeud](4) at (1.5,1) {};
			\node[jolinoeud](5) at (1,1.75) {};
			\draw (1)--(0)--(2);
			\draw (1)--(2)--(4)--(3)--(1);
			\draw (4)--(5)--(2);
			\draw[line width=0.75mm] (2)to(4);
			\end{tikzpicture}
		};
	\end{tikzpicture}
	\caption{The graph $H_5$, a D2C graph with a dominating edge (in bold).}
	\label{fig-h5}
\end{figure}
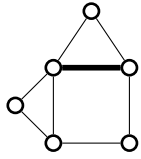

We have performed a computer search on all graphs of order up to $11$. This search has found exactly thirteen non-bipartite D2C graphs that are not members of $\F$ but nevertheless reach the bound $\lfloor(n-1)^2/4\rfloor+1$ (see Figure~\ref{fig:thirteen-graphs}).\footnote{Note that the complete list of D2C graphs of order at most $7$ was given in~\cite{MacDougall-Eggleton}; Graph (m) is studied in~\cite{plesnik2} as a planar D2C graph with every edge in a triangle.} These graphs have order $7$, $8$ or $9$: there are no examples of order $10$ and $11$. Only three of these graphs have no dominating edge.

\begin{figure}[!ht]
  \centering
	\begin{tikzpicture}
		\node (g7-1) at (-3,0) {
			\begin{tikzpicture}
			\node[jolinoeud](0) at (-1,0.5) {};
			\node[jolinoeud](1) at (-0.5,0) {};
			\node[jolinoeud](2) at (-0.5,1) {};
			\node[jolinoeud](3) at (0.5,0) {};
			\node[jolinoeud](4) at (0.5,1) {};
			\node[jolinoeud](5) at (0,1.35) {};
			\node[jolinoeud](5p) at (0,1.75) {};
			\draw (1)--(0)--(2);
			\draw (1)--(2)--(4)--(3)--(1);
			\draw (4)--(5)--(2);
			\draw (4)--(5p)--(2);
			\draw[line width=0.75mm] (2)to(4);
                        \node at (0,-0.5) {(a)};
			\end{tikzpicture}
		};
		\node (g7-2) at (0,0) {
			\begin{tikzpicture}
			\node[jolinoeud](0) at (-1,0.5) {};
			\node[jolinoeud](1) at (-0.5,0) {};
			\node[jolinoeud](2) at (-0.5,1) {};
			\node[jolinoeud](3) at (0.5,0) {};
			\node[jolinoeud](4) at (0.5,1) {};
			\node[jolinoeud](5) at (0,1.75) {};
			\node[jolinoeud](6) at (0,0.5) {};
			\draw (1)--(0)--(2);
			\draw (1)--(2)--(4)--(3)--(1);
			\draw (4)--(5)--(2);
			\draw (4)--(6)--(1);
			\draw[line width=0.75mm] (2)to(4);
                        \node at (0,-0.5) {(b)};
			\end{tikzpicture}
		};
		\node (g7-3) at (3,0) {
			\begin{tikzpicture}
			\node[jolinoeud](0) at (-1,0.5) {};
			\node[jolinoeud](1) at (-0.5,0) {};
			\node[jolinoeud](2) at (-0.5,1) {};
			\node[jolinoeud](3) at (0.5,0) {};
			\node[jolinoeud](4) at (0.5,1) {};
			\node[jolinoeud](5) at (0,1.75) {};
			\node[jolinoeud](6) at (0,0.5) {};
			\draw (1)--(0)--(2);
			\draw (1)--(2)--(4)--(3)--(1);
			\draw (4)--(5)--(2);
			\draw (2)--(6)--(3);
			\draw[line width=0.75mm] (2)to(4);
                        \node at (0,-0.5) {(c)};
			\end{tikzpicture}
		};
		\node (g8-1) at (-3,-3) {
			\begin{tikzpicture}
			\node[jolinoeud](0) at (-1,1) {};
			\node[jolinoeud](1) at (-0.5,0) {};
			\node[jolinoeud](2) at (-0.5,1) {};
			\node[jolinoeud](3) at (0.5,0) {};
			\node[jolinoeud](4) at (0.5,1) {};
			\node[jolinoeud](5) at (0,1.5) {};
			\node[jolinoeud](6) at (-0.5,2) {};
			\node[jolinoeud](7) at (0.5,2) {};
			\draw (1)--(0)--(2);
			\draw (1)--(2)--(4)--(3)--(1);
			\draw (4)--(5)--(2);
			\draw (0)--(6)--(7)--(5)--(6);
                        \draw[bend right] (3) to (7);
                        \node at (0,-0.5) {(d)};
			\end{tikzpicture}
		};

		\node (g8-2) at (0,-3) {
			\begin{tikzpicture}
			\node[jolinoeud](0) at (-1,0.5) {};
			\node[jolinoeud](1) at (-0.5,0) {};
			\node[jolinoeud](2) at (-0.5,1) {};
			\node[jolinoeud](3) at (0.5,0) {};
			\node[jolinoeud](4) at (0.5,1) {};
			\node[jolinoeud](5) at (0,1.5) {};
			\node[jolinoeud](6) at (-1,1.5) {};
			\node[jolinoeud](7) at (1,2) {};
			\draw (1)--(0)--(2);
			\draw (1)--(2)--(4)--(3)--(1);
			\draw (4)--(5)--(2);
			\draw (0)--(6)--(2);
			\draw (5)--(6)--(7)--(3);
                        \node at (0,-0.5) {(e)};
			\end{tikzpicture}
		};
		\node (g8-3) at (3,-3) {
			\begin{tikzpicture}
			\node[jolinoeud](0) at (0,-0.25) {};
                        \node[jolinoeud](7) at (0,1.75) {};
			\node[jolinoeud](1) at (-0.33,0.5) {};
			\node[jolinoeud](2) at (-0.33,1) {};
			\node[jolinoeud](3) at (-1,0.75) {};
			\node[jolinoeud](4) at (0.33,0.5) {};
			\node[jolinoeud](5) at (0.33,1) {};
			\node[jolinoeud](6) at (1,0.75) {};
			\draw (0)--(3)--(1)--(0)--(4)--(6)--(0);
			\draw (1)--(2)--(3)--(7)--(6)--(5)--(4);
			\draw (5)--(2);
                        \node at (0,-0.75) {(f)};
			\end{tikzpicture}
		};
		\node (g8-4) at (-6,-6) {
			\begin{tikzpicture}
			\node[jolinoeud](0) at (-1,0.5) {};
			\node[jolinoeud](1) at (-0.5,0) {};
			\node[jolinoeud](2) at (-0.5,1) {};
			\node[jolinoeud](3) at (0.5,0) {};
			\node[jolinoeud](4) at (0.5,1) {};
			\node[jolinoeud](5) at (0,1.75) {};
			\node[jolinoeud](6) at (0.2,0.3) {};
			\node[jolinoeud](7) at (-0.1,0.6) {};
			\draw (1)--(0)--(2);
			\draw (1)--(2)--(4)--(3)--(1);
			\draw (4)--(5)--(2);
			\draw (1)--(6)--(4);
			\draw (3)--(6)--(7)--(2);
			\draw[line width=0.75mm] (2)to(4);
                        \node at (0,-0.5) {(g)};
			\end{tikzpicture}
		};
 		\node (g8-5) at (-3,-6) {
			\begin{tikzpicture}
			\node[jolinoeud](1) at (-0.75,0) {};
			\node[jolinoeud](2) at (-0.75,1) {};
			\node[jolinoeud](3) at (0.75,0) {};
			\node[jolinoeud](4) at (0.75,1) {};
			\node[jolinoeud](5) at (0,1.75) {};
			\node[jolinoeud](6) at (-0.375,0.5) {};
			\node[jolinoeud](7) at (0,0.5) {};
			\node[jolinoeud](8) at (0.375,0.5) {};
			\draw (1)--(2)--(4)--(3)--(1);
			\draw (4)--(5)--(2);
			\draw (1)--(6)--(2)--(7)--(6);
			\draw (7)--(8)--(3);
			\draw (4)--(8);
			\draw[line width=0.75mm] (2)to(4);
                        \node at (0,-0.5) {(h)};
			\end{tikzpicture}
		};
 		\node (g8-6) at (0,-6) {
			\begin{tikzpicture}
			\node[jolinoeud](0) at (-1,0.5) {};
			\node[jolinoeud](1) at (-0.5,0) {};
			\node[jolinoeud](2) at (-0.5,1) {};
			\node[jolinoeud](3) at (0.5,0) {};
			\node[jolinoeud](4) at (0.5,1) {};
			\node[jolinoeud](5) at (0,1.75) {};
			\node[jolinoeud](6) at (-0.25,0.5) {};
			\node[jolinoeud](7) at (.25,0.5) {};
			\draw (1)--(0)--(2);
			\draw (1)--(2)--(4)--(3)--(1);
			\draw (4)--(5)--(2);
			\draw (2)--(6)--(7)--(4);
			\draw (3)--(6) (7)--(1);
			\draw[line width=0.75mm] (2)to(4);
                        \node at (0,-0.5) {(i)};
			\end{tikzpicture}
		};
 		\node (g8-7) at (3,-6) {
			\begin{tikzpicture}
			\node[jolinoeud](1) at (-0.5,0) {};
			\node[jolinoeud](2) at (-0.5,1) {};
			\node[jolinoeud](3) at (0.5,0) {};
			\node[jolinoeud](4) at (0.5,1) {};
			\node[jolinoeud](5) at (0,1.35) {};
			\node[jolinoeud](6) at (0,1.75) {};
			\node[jolinoeud](7) at (1.25,1) {};
			\node[jolinoeud](8) at (0.75,0.75) {};
			\draw (1)--(2)--(4)--(3)--(1);
			\draw (4)--(5)--(2);
			\draw (2)--(6)--(7)--(4)--(8)--(3);
			\draw (3)--(7)--(8);
			\draw[line width=0.75mm] (2)to(4);
                        \node at (0,-0.5) {(j)};
			\end{tikzpicture}
		};
 		\node (g8-8) at (6,-6) {
			\begin{tikzpicture}
			\node[jolinoeud](1) at (-0.5,0) {};
			\node[jolinoeud](2) at (0.5,0) {};
			\node[jolinoeud](3) at (-1,0.875) {};
			\node[jolinoeud](4) at (1,0.875) {};
			\node[jolinoeud](5) at (-0.5,1.75) {};
			\node[jolinoeud](6) at (0.5,1.75) {};
			\node[jolinoeud](7) at (0,0.5) {};
			\node[jolinoeud](8) at (0,1.25) {};
			\draw (2)--(1)--(3)--(5)--(6)--(4)--(2);
			\draw (2)--(8)--(7)--(4);
			\draw (3)--(7) (1)--(8);
			\draw (5)--(8)--(6);
			\draw[line width=0.75mm] (7)to(8);
                        \node at (0,-0.5) {(k)};
			\end{tikzpicture}
		};
 		\node (g9-1) at (-1.5,-9) {
			\begin{tikzpicture}
			\node[jolinoeud](0) at (0,0) {};
			\node[jolinoeud](1) at (0,0.66) {};
			\node[jolinoeud](2) at (0,1.33) {};
			\node[jolinoeud](3) at (0,2) {};
			\node[jolinoeud](4) at (1,2) {};
			\node[jolinoeud](5) at (1,1.33) {};
			\node[jolinoeud](6) at (1,0.66) {};
			\node[jolinoeud](7) at (0.5,1.33) {};
			\node[jolinoeud](8) at (-0.75,0.75) {};
			\draw (0)--(1)--(2)--(3)--(4)--(5)--(6)--(0);
			\draw (0)--(8)--(3);
			\draw (5)--(7)--(4);
			\draw (2)--(7)--(1)--(6);
                        \draw[bend right] (3) to (0);
			\draw[line width=0.75mm] (7)to(0);
                        \node at (0,-0.5) {(l)};
			\end{tikzpicture}
		};
 		\node (g9-2) at (1.5,-9) {
			\begin{tikzpicture}
			\node[jolinoeud](0) at (0,1.33) {};
			\node[jolinoeud](1) at (0,2) {};
			\node[jolinoeud](2) at (0,0.66) {};
			\node[jolinoeud](3) at (0,0) {};
			\node[jolinoeud](4) at (0.66,1.33) {};
			\node[jolinoeud](5) at (0.66,0.66) {};
			\node[jolinoeud](6) at (-0.66,1.33) {};
			\node[jolinoeud](7) at (-0.66,0.66) {};
			\node[jolinoeud](8) at (1.66,1) {};
			\draw (0)--(1)--(4)--(0)--(5)--(4);
			\draw (3)--(5)--(2)--(3)--(7)--(2);
			\draw (0)--(7)--(6)--(0);
			\draw (6)--(1)--(8)--(3);
			\draw[line width=0.75mm] (1) .. controls ++(-1.25,0) and ++(-1.25,0) ..(3);
                        \node at (0,-0.5) {(m)};
			\end{tikzpicture}
		};
        \end{tikzpicture}
	\caption{The thirteen non-bipartite D2C graphs of order $n\leq 11$ with $\lfloor(n-1)^2/4\rfloor+1$ edges that are not in the family $\F$. Bold edges are dominating. Only the graphs (d), (e) and (f) have no dominating edge.}
	\label{fig:thirteen-graphs}
\end{figure}
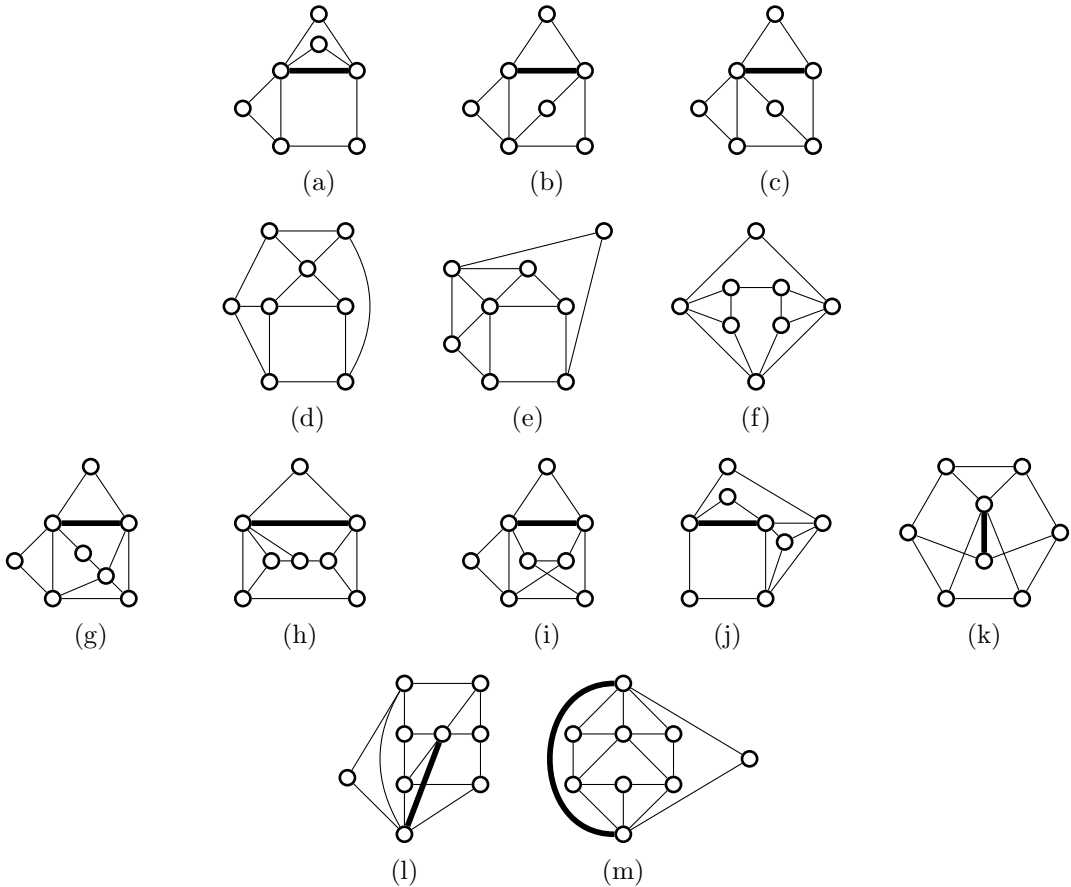

Closer scrutiny and the aforemetioned computer search suggest that $H_5$ could well be the only small non-bipartite D2C exception to the bound of Conjecture~\ref{conj:BHHH} (this was also suggested in~\cite{Barefoot}). Moreover, it seems likely that only the thirteen graphs not in $\F$ from Figure~\ref{fig:thirteen-graphs} reach this bound. This leads us to propose a stronger version of Conjecture~\ref{conj:BHHH}, as follows.

\begin{conjecture}\label{conj}
Let $G$ be a non-bipartite D2C graph of order $n$. If $G$ is not $H_5$, then $G$ has at most $\left\lfloor(n-1)^2/4\right\rfloor+1$ edges, with equality if and only if $G$ belongs to $\F$ or is one of the thirteen graphs from Figure~\ref{fig:thirteen-graphs}.
\end{conjecture}



In this paper, we focus on the case of non-bipartite D2C graphs with a dominating edge. Note that the graphs in $\F$ do not have a dominating edge. As a piece of evidence towards Conjecture~\ref{conj}, our main theorem shows that the bound of Conjecture~\ref{conj:MS} is not tight for this case:

\begin{theorem}\label{thm:main}
Let $G$ be a non-bipartite D2C graph with $n$ vertices having a dominating edge. If $G$ is not $H_5$, then $G$ has at most $\left\lfloor n^2/4\right\rfloor-2$ edges.
\end{theorem}

Along the way, we will also give some stronger bounds for special cases.
Our proof technique is an extension of the technique introduced by Hanson and Wang in~\cite{HW03}, where the bound of Conjecture~\ref{conj:MS} was proved for graphs with a dominating edge. The full statement of Conjecture~\ref{conj:MS} was later proved for this case in the papers~\cite{diam3,correction,Wang-arxiv} using a different method, but the proof is quite involved. As a side result, our new approach also gives a simpler and shorter proof of Conjecture~\ref{conj:MS} for graphs with a dominating edge.

We also study non-bipartite D2C graphs with maximum degree $n-2$ (note that the only D2C graphs with maximum degree $n-1$ are stars). It turns out that these graphs can be described precisely, and they all have a dominating edge.

We prove Theorem~\ref{thm:main} (and a stronger version for special cases) in Section~\ref{sec:main} and obtain, in passing, a shorter proof of Conjecture~\ref{conj:MS} for graphs with a dominating edge. We then characterize D2C graphs with maximum degree $n-2$ in Section~\ref{sec:highdegree}. Finally, we conclude in Section~\ref{sec:conclu}.

\section{Conjecture~\ref{conj:MS} and beyond for graphs with a dominating edge}\label{sec:main}

In this section, we further develop a proof technique of Hanson and Wang~\cite{HW03}, who showed the bound of Conjecture~\ref{conj:MS} for D2C graphs having a dominating edge. Their result was extended (using a different technique) by Haynes et al. and Wang (see the series of papers~\cite{diam3,correction,Wang-arxiv}) to prove the full conjecture for this class. Nevertheless, it turns out that when excluding bipartite graphs, the Murty-Simon bound is not tight for this class of graphs. To show this, we use the original idea of Hanson and Wang~\cite{HW03} and extend it by a finer analysis. To demonstrate the potential strength of this technique, we give, along the way, a shorter proof of Conjecture~\ref{conj:MS} for graphs with a dominating edge.

In order to have a smoother presentaion, we split this section into several parts. We start with establishing the general setting of the proof technique in Section~\ref{sec:prelim}. Then, we prove some general lemmas in Sections~\ref{sec:lemmas-Puvnonempty} and~\ref{sec:lemmas-Puvempty}. We use some of these lemmas to give our new proof of Conjecture~\ref{conj:MS} for graphs with a dominating edge in Section~\ref{sec:MSproof}. Then, in Section~\ref{sec:strong-theorem}, we prove a stronger bound than the one of Theorem~\ref{thm:main} for some special cases. Finally, in Section~\ref{sec:mainproof}, we conlclude the proof of Theorem~\ref{thm:main}.

\subsection{Preliminaries: notations and setting for the proofs}\label{sec:prelim}

Let us first fix our notation. Given a vertex $x$ from a graph $G$, we denote by $N(x)$ and $N[x]$ the open and closed neighbourhoods of $x$, respectively. An edge between vertices $x$ and $y$ is denoted $xy$, while a non-edge between $x$ and $y$ is denoted $\overline{xy}$. An \emph{oriented graph} is a graph where edges have been given an orientation; oriented edges are called \emph{arcs}. If $x$ is oriented towards $y$, we denote the arc from $x$ to $y$ by $\overrightarrow{xy}$. In an oriented graph, we denote by $N^+(x)$, $N^+[x]$, $N^-(x)$ and $N^-[x]$ the out-neighbourhood, closed out-neighbourhood, in-neighbourhood, and closed in-neighbourhood of vertex $x$. In an oriented graph, a \emph{directed cycle} is a cycle such that all arcs are oriented in the same cyclic direction. We say that a \emph{source} is a vertex $s$ with $N^-(s)=\emptyset$ and $N^+(s) \neq \emptyset$ while a \emph{sink} is a vertex $t$ with $N^+(t)=\emptyset$ and $N^-(t) \neq \emptyset$ (we consider that an isolated vertex is neither a source nor a sink). A triangle on an oriented graph is \emph{transitive} if it induces a subgraph with a source and a sink.

We start with the following definition, which is fundamental to our study.

\begin{definition}
	\label{def:criticalEdge}
	Let $G$ be a D2C graph. An edge $uv \in E(G)$ is \emph{critical} for a pair of vertices $\{x,y\}$ if the only path of length 1 or 2 from $x$ to $y$ uses the edge $uv$.
\end{definition}

The following observation is easy but important.

\begin{observation}\label{obs:critical}
  An edge $xy$ in a D2C graph is critical for a pair $\{x,z\}$ with $z\in N[y]\setminus \{x\}$ or $\{y,z\}$ with $z\in N[x]\setminus \{y\}$.
\end{observation}

We are ready to describe the setting for the proofs of this section. Let $G(V,E)$ be a D2C graph with $n$ vertices and $m$ edges, and let $uv$ be a dominating edge of $G$. We split the other vertices of $G$ into four sets (see Figure~\ref{fig-structureD2CAreteDom} for an illustration):

	\begin{enumerate}
		\item $P_{uv}=\{x~|~uv \text{ is critical for the pair } \{x,v\} \text{ or } \{x,u\}\}$
		\item $S_{uv}=\{x~|~x \in N(u) \text{ and } x \in N(v)\}$
		\item $S_u=\{x~|~x \in N(u)\setminus(P_{uv} \cup S_{uv})\}$
		\item $S_v=\{x~|~x \in N(v)\setminus(P_{uv} \cup S_{uv})\}$
	\end{enumerate}

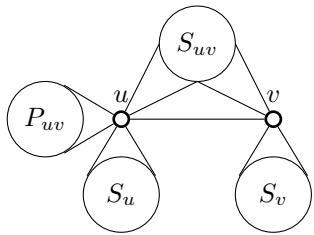
\begin{figure}[ht!]
	\centering
	\begin{tikzpicture}
	\node[jolinoeud](u) at (0,0) {};
	\node[jolinoeud](v) at (2,0) {};
	\draw (u)--(v);
	\draw (0,0.3) node {$u$};
	\draw (2,0.3) node {$v$};
	\draw[fill=white] (-1,0) circle (0.5);
	\draw (-1,0) node {$P_{uv}$};
	\draw (u)--(-0.75,0.45);
	\draw (u)--(-0.75,-0.45);
	\draw[fill=white] (1,1) circle (0.5);
	\draw (1,1) node {$S_{uv}$};
	\draw (u)--(0.5,1);
	\draw (u)--(1,0.5);
	\draw (v)--(1,0.5);
	\draw (v)--(1.5,1);
	\draw[fill=white] (0,-1) circle (0.5);
	\draw (0,-1) node {$S_u$};
	\draw (u)--(-0.45,-0.75);
	\draw (u)--(0.45,-0.75);
	\draw[fill=white] (2,-1) circle (0.5);
	\draw (2,-1) node {$S_v$};
	\draw (v)--(1.55,-0.75);
	\draw (v)--(2.45,-0.75);
	\end{tikzpicture}
	\caption{The structure of a D2C graph with the dominating edge $uv$ (the only edges that are depicted are those incident with $u$ or $v$). Lemma~\ref{clm:PuvSeesU} allows us to represent all vertices in $P_{uv}$ as adjacent to $u$.}
	\label{fig-structureD2CAreteDom}
\end{figure}

We have the following fact.

\begin{lemma}
		\label{clm:PuvSeesU}
Either $P_{uv}\cap N(u)=\emptyset$, or $P_{uv} \cap N(v) = \emptyset$ (or both).
	\end{lemma}
	\begin{proof}
		Let $x \in P_{uv} \cap N(u)$, and assume by contradiction that there is $y \in P_{uv}\cap N(v)$. We have $xy \not\in E$ since $uv$ is critical for the pairs $(u,y)$ and $(v,x)$. Thus, $x$ and $y$ have a common neighbour $z$. However, since $uv$ is a dominating edge, $z \in N(u) \cup N(v)$. Suppose without loss of generality that $z \in N(u)$, then there are two paths of length~$2$ between $u$ and $y$: one going through $v$ and one going through $z$. Thus, the edge $uv$ is not critical for the pair $\{u,y\}$, a contradiction.
	\end{proof}

Because of Lemma~\ref{clm:PuvSeesU}, in the whole section, without loss of generality, we will always assume that $P_{uv} \cap N(v) = \emptyset$. We next prove the following lemma.

\begin{lemma}
  The following properties hold.
  \begin{itemize}
  \item[(a)] There is no edge between $P_{uv}$ and $N(v) \setminus \{u\}$.\label{clm:Puv-Sv}
  \item[(b)] If $P_{uv}=\emptyset$, then $S_{uv}=\emptyset$.\label{clm:PuvEmptySuvEmpty}
  \item[(c)] If $S_{uv}=\emptyset$, then every vertex in $S_u$ (resp. $S_v$) has a neighbour in $S_v$ (resp. $S_u$).\label{clm:everySuHasNeighbInSv}
  \item[(d)] If $P_{uv}=\emptyset$, then every vertex in $S_u$ (resp. $S_v$) that has at least one neighbour in $S_u$ (resp. $S_v$) has a non-neighbour in $S_v$ (resp. $S_u$).\label{clm:everySuHasNonNeighbInSv}
  \end{itemize}
\end{lemma}
\begin{proof}(a) Let $x\in P_{uv}$ and assume by contradiction that there is $y \in N(v) \setminus \{u\}$ such that $xy \in E$. Then there are two paths of length 2 between $x$ and $v$: one going through $u$ and one going through $y$. Thus, the edge $uv$ is not critical for the pair $\{v,x\}$, a contradiction.

  (b) If $P_{uv}=\emptyset$, then the edge $uv$ can only be critical for the pair $\{u,v\}$. This implies that $u$ and $v$ have no common neighbour, that is, $S_{uv}=\emptyset$.

  (c) Assume by contradiction that there is a vertex $x \in S_u$ (without loss of generality) such that $N(x) \cap S_v = \emptyset$. Then $N(x) \cap N(v) = \{u\}$ since $S_{uv} = \emptyset$. This implies that the edge $uv$ is critical for the pair $\{x,v\}$, and thus $x \in P_{uv}$, a contradiction.

(d) Assume by contradiction that there is a vertex $x \in S_u$ (without loss of generality) such that $S_v \subset N(x)$. Then, the edge $ux$ is not critical. Indeed, it cannot be critical for the pair $\{u,x\}$ since $x$ has a neighbour in $S_u$. It cannot be critical for a pair $\{x,y\}$ with $y \in S_u$: since $P_{uv}=\emptyset$ we have $S_{uv}=\emptyset$ by (b), and by (c), $y$ has a neighbour in $S_v$. So, there is a path of length~2 from $x$ to $y$ going through $S_v$. Finally, it cannot be critical for a pair $\{u,y\}$ with $y \in N(x)$ since every neighbour of $x$ is either in $S_u$ (thus, a neighbour of $u$) or in $S_v$ (and a neighbour of $v$). Observation~\ref{obs:critical} ensures that we considered all the cases, and reached a contradiction which proves the claim.
\end{proof}

Following the proof of Hanson and Wang~\cite{HW03}, we will next partition the vertices of $G$ into two parts $X$ and $Y$, and prove that every edge within $X$ or within $Y$ can be assigned injectively to a non-edge between $X$ and $Y$. This will prove that $G$ has at most as many edges as the complete bipartite graph with parts $X$ and $Y$.
	
	We define the partition as follows:
	\begin{enumerate}
		\item $X := \{v\} \cup S_u \cup P_{uv} \cup S_{uv}$
		\item $Y := \{u\} \cup S_v$
	\end{enumerate}

	\begin{lemma}
		\label{clm:assignment}
		For every edge $ab \in E(X)$ (resp. $E(Y)$), there exists $c \in Y$ (resp. $X$) such that $ab$ is critical for either the pair $\{a,c\}$ or the pair $\{b,c\}$.
	\end{lemma}

	\begin{proof}
		Assume without loss of generality that $a,b \in X$. This implies that both $a$ and $b$ are neighbours of $u$. Then, the edge $ab$ cannot be critical for the pair $\{a,b\}$.  Without loss of generality, we assume $ab$ it is critical for $\{b,c\}$, where $c \in N(a) \setminus N(b)$. However, if $c \in X$, then $b$ and $c$ are neighbours of $u$, and then the edge $ab$ is not critical for this pair.
		Hence, $c \in Y$.
	\end{proof}

	We use Lemma~\ref{clm:assignment} to define a function $f$ assigning the edges of $E(X)$ and $E(Y)$ to non-edges between $X$ and $Y$, as follows.
	For every edge $e \in E(X)$ (resp. $E(Y)$), we select one vertex $c \in Y$ (resp. $X$) such that $e$ is critical for the pair $\{b,c\}$, where $b \in e$ (such vertex $c$ exists by Lemma~\ref{clm:assignment}). We let $f(e)=\overline{bc}$. Note that $f$ is well-defined, since $e$ is critical for the pair $\{b,c\}$ and thus $bc \not\in E$. This construction is depicted in Figure~\ref{fig-functionF}.
	
	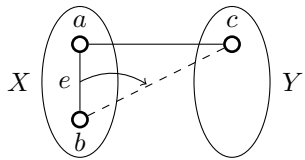
\begin{figure}[ht!]
		\centering
		\begin{tikzpicture}
		\node[jolinoeud](y) at (0,0) {};
		\node[jolinoeud](z) at (0,1) {};
		\node[jolinoeud](x) at (2,1) {};
		\draw (0,0.5) ellipse (0.5 and 1);
		\draw (2,0.5) ellipse (0.5 and 1);
		\draw (-0.2,0.5) node {$e$};
		\draw (0,-0.3) node {$b$};
		\draw (0,1.3) node {$a$};
		\draw (2,1.3) node {$c$};
		\draw[dashed] (x)--node[midway,below](0){}(y);
		\draw (x)--(z);
		\draw (y)--(z);
		\draw (-0.8,0.5) node {$X$};
		\draw (2.8,0.5) node {$Y$};
		\draw[->,bend left] (0,0.5)to(0);
		\end{tikzpicture}
		\caption{The construction of the function $f$. The edge $ e = ab$ is critical for the pair $\{b,c\}$, thus $f(e)=\overline{bc}$.}
		\label{fig-functionF}
	\end{figure}
	
	\begin{lemma}\label{clm:assignIsInj}
		The function $f$ is injective.
	\end{lemma}

	\begin{proof}
		Assume by contradiction that $f$ is not injective. Without loss of generality, let $\overline{bc}$ be the non-edge between $X$ and $Y$ such that there are two edges $e$ and $e'$ in $E(X) \cup E(Y)$ verifying $f(e)=f(e')=\overline{bc}$.
		
		By definition of $f$, both $e$ and $e'$ are critical for the pair $\{b,c\}$. This implies that $e$ and $e'$ form the unique path of length~$2$ from $b$ to $c$. Thus, one of $e$ or $e'$ is in $X \times Y$, a contradiction.
	\end{proof}

        We saw in Lemma~\ref{clm:assignIsInj} that $f$ is injective. Moreover, we will show later that, if $G$ is not bipartite, then $f$ is not surjective. We call any non-edge in $X\times Y$ that has no preimage by $f$, an \emph{$f$-free non-edge}. We also let $\free(f)$ be the number of $f$-free non-edges.


\begin{lemma}\label{clm:MSBound}
  We have $m=\left\lfloor\frac{n^2-||X|-|Y||^2}{4}\right\rfloor-\free(f)\leq \left\lfloor n^2/4 \right\rfloor-\free(f)$.
\end{lemma}
\begin{proof}
  By the injectivity of $f$ (Lemma~\ref{clm:assignIsInj}) and the definition of $\free(f)$, there are exactly $\free(f)$ more non-edges between $X$ and $Y$ than edges inside $X$ and inside $Y$. Thus, $G$ has exactly $|X||Y|-\free(f)$ edges.
  
Without loss of generality, we assume that $|X|\leq |Y|$ and we pose $\Delta=||X|-|Y||$. By the above paragraph, we have $m=|X||Y|-\free(f)=|X|(n-|X|)-\free(f)$.
Since $|X|+\Delta=|Y|=n-|X|$, this implies that $|X|=\frac{n-\Delta}{2}$. In particular, we now have:
	        
		\begin{eqnarray*}
		m & = & |X|(n-|X|)-\free(f) \\
		& = & \frac{n-\Delta}{2} \left(n-\frac{n-\Delta}{2}\right)-\free(f) \\
		& = & \frac{(n-\Delta)(n+\Delta)}{4}-\free(f) \\
		& = & \frac{n^2-\Delta^2}{4}-\free(f)
		\end{eqnarray*}

Because $m$ and $\free(f)$ are integers, we have $m=\left\lfloor\frac{n^2-\Delta^2}{4}\right\rfloor-\free(f)$. Moreover, since $\Delta\geq 0$, we obtain that $m\leq\left\lfloor n^2/4 \right\rfloor-\free(f)$.
\end{proof}

Lemma~\ref{clm:MSBound} implies that $G$ has at most $\left\lfloor n^2/4 \right\rfloor$ edges: this is the result of Hanson and Wang~\cite{HW03}. It will require some more effort to prove the whole Conjecture~\ref{conj:MS}: our aim will be to show that $\free(f)$ has at least a certain size. First, we will prove some general lemmas.

\subsection{Preliminaries for the case $P_{uv}\neq\emptyset$}\label{sec:lemmas-Puvnonempty}

We now prove a useful lemma about $P_{uv}$, which is illustrated in Figure~\ref{fig-PuvEstRelou}.

\begin{lemma}\label{clm:P_{uv}}\label{eq:Puv}
  Let $p$ be a vertex in $P_{uv}$, and let $S_v(p)$ be the set of vertices $x \in S_v$ such that the non-edge $\overline{px}$ is not $f$-free. Then, for each vertex $x\in S_v(p)$, there is a vertex $m(x)$ in $S_u \cap N(p)$ such that $f(pm(x))=\overline{px}$. Denote by $S_u(p)$ the set of vertices $y$ of $S_u$ such that $y=m(x)$ for some vertex $x$ of $S_v(p)$. Then, the following holds.
  \begin{itemize}
  \item[(a)] We have $|S_u(p)|=|S_v(p)|$ (that is, $m$ is injective).
  \item[(b)] The only edges in $S_u(p)\times S_v(p)$ are those of the form $xm(x)$.
  \item[(c)] For any two vertices $x,y$ of $S_v(p)$, if one of the edges $xy$ or $m(x)m(y)$ exists, then one of the non-edges $\overline{xm(y)}$ and $\overline{ym(x)}$ is $f$-free. If both edges $xy$ and $m(x)m(y)$ exist, then both non-edges $\overline{xm(y)}$ and $\overline{ym(x)}$ are $f$-free.
    \item[(d)] We have $|S_u|\geq |S_v|-\free(f)$.
  \end{itemize}
\end{lemma}
\begin{proof}
  Let $x\in S_v(p)$. By Lemma~\ref{clm:Puv-Sv}(a), $p$ has no neighbour neither in $S_v$ nor in $S_{uv}$. Thus, $p$ and $x$ have a common neighbour, $q$, in $S_u$, and $f(pq)=\overline{px}$. We let $q=m(x)$. Now, if for some pair $x,y$ of distinct vertices of $S_v(p)$, we had $m(x)=m(y)$, then one of the non-edges $\overline{px}$ and $\overline{py}$ would be $f$-free (since both can only be assigned to $pq$ by $f$), a contradicton. Thus, $|S_u(p)|=|S_v(p)|$ and (a) is true.

  Moreover, there is no edge $xm(y)$ for two distinct vertices $x,y$ in $S_v(p)$, since otherwise $p$ and $x$ would have two common neighbours ($m(x)$ and $m(y)$), contradicting the fact that $f(pm(x))=\overline{px}$. Thus, (b)~holds.

  Finally, assume that there is an edge $xy$ in $S_v(p)$ (the proof is the same for the edge $m(x)m(y)$). Then, both non-edges $\overline{xm(y)}$ and $\overline{ym(x)}$ can only be assigned to the edge $xy$, so one of them is $f$-free. If we have both edges $xy$ and $m(x)m(y)$, then both endpoints of $\overline{xm(y)}$ and $\overline{ym(x)}$ have two common neighbours, so both are $f$-free, and (c)~is true.

To prove~(d), we let $\free(P_{uv},f)$ be the number of $f$-free non-edges incident with a vertex of $P_{uv}$. Let $p$ be some vertex $p$ of $P_{uv}$. By the previous parts of the lemma, we have:

\begin{equation*}
|S_u|\geq|S_u(p)|=|S_v(p)|\geq |S_v|-\free(P_{uv},f)\geq |S_v|-\free(f),
\end{equation*}
which completes the proof of (d).
\end{proof}

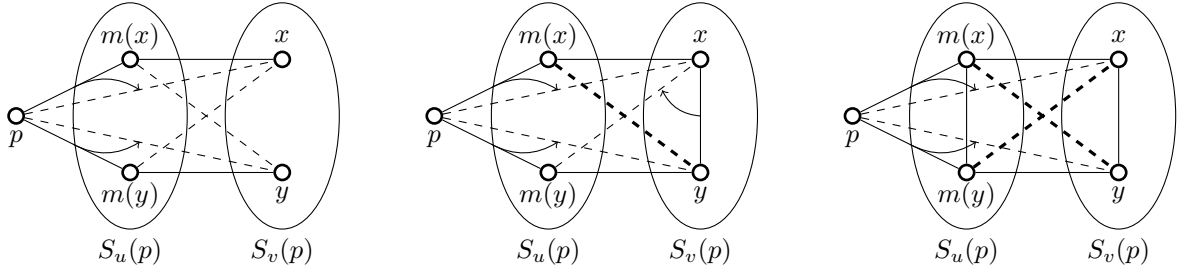
\begin{figure}[ht!]
	\centering
	\begin{tikzpicture}	
	\node (puv) at (0,0) {
		\begin{tikzpicture}
		\draw (0,0) ellipse (0.75 and 1.5);
		\draw (2,0) ellipse (0.75 and 1.5);
		\node[jolinoeud] (0) at (0,-0.75) {};
		\node[jolinoeud] (1) at (0,0.75) {};
		\node[jolinoeud] (2) at (2,-0.75) {};
		\node[jolinoeud] (3) at (2,0.75) {};
		\draw (0,-1.05) node {$m(y)$};
		\draw (0,1.05) node {$m(x)$};
		\draw (2,1.05) node {$x$};
		\draw (2,-1.05) node {$y$};
		\node[jolinoeud] (p) at (-1.5,0) {};
		\draw (-1.5,-0.3) node {$p$};
		\draw (p)--(0);
		\draw (p)--(1);
		\draw (2)--(0);
		\draw (3)--(1);
		\draw[dashed](p)--node[midway,above](m1){}(2);
		\draw[dashed](p)--node[midway,below](m2){}(3);
		\draw[->,bend left] (-0.75,0.375)to(m2);
		\draw[->,bend right] (-0.75,-0.375)to(m1);
		\draw[dashed] (0)--(3);
		\draw[dashed] (1)--(2);
		\draw (0,-1.8) node {$S_u(p)$};
		\draw (2,-1.8) node {$S_v(p)$};
		\end{tikzpicture}
	};
	
	\node (puv2) at (5.5,0) {
		\begin{tikzpicture}
		\draw (0,0) ellipse (0.75 and 1.5);
		\draw (2,0) ellipse (0.75 and 1.5);
		\node[jolinoeud] (0) at (0,-0.75) {};
		\node[jolinoeud] (1) at (0,0.75) {};
		\node[jolinoeud] (2) at (2,-0.75) {};
		\node[jolinoeud] (3) at (2,0.75) {};
		\draw (0,-1.05) node {$m(y)$};
		\draw (0,1.05) node {$m(x)$};
		\draw (2,1.05) node {$x$};
		\draw (2,-1.05) node {$y$};
		\node[jolinoeud] (p) at (-1.5,0) {};
		\draw (-1.5,-0.3) node {$p$};
		\draw (p)--(0);
		\draw (p)--(1);
		\draw (2)--(0);
		\draw (3)--(1);
		\draw (2)--(3);
		\draw[dashed](p)--node[midway,above](m1){}(2);
		\draw[dashed](p)--node[midway,below](m2){}(3);
		\draw[->,bend left] (-0.75,0.375)to(m2);
		\draw[->,bend right] (-0.75,-0.375)to(m1);
		\draw[dashed] (0)--node[near end,above](m3){}(3);
		\draw[dashed,very thick] (1)--(2);
		\draw[->,bend left] (2,0)to(m3);
		\draw (0,-1.8) node {$S_u(p)$};
		\draw (2,-1.8) node {$S_v(p)$};
		\end{tikzpicture}
	};
	
	\node (puv3) at (11,0) {
		\begin{tikzpicture}
		\draw (0,0) ellipse (0.75 and 1.5);
		\draw (2,0) ellipse (0.75 and 1.5);
		\node[jolinoeud] (0) at (0,-0.75) {};
		\node[jolinoeud] (1) at (0,0.75) {};
		\node[jolinoeud] (2) at (2,-0.75) {};
		\node[jolinoeud] (3) at (2,0.75) {};
		\draw (0,-1.05) node {$m(y)$};
		\draw (0,1.05) node {$m(x)$};
		\draw (2,1.05) node {$x$};
		\draw (2,-1.05) node {$y$};
		\node[jolinoeud] (p) at (-1.5,0) {};
		\draw (-1.5,-0.3) node {$p$};
		\draw (p)--(0);
		\draw (p)--(1);
		\draw (1)--(0);
		\draw (2)--(0);
		\draw (3)--(1);
		\draw (2)--(3);
		\draw[dashed](p)--node[midway,above](m1){}(2);
		\draw[dashed](p)--node[midway,below](m2){}(3);
		\draw[->,bend left] (-0.75,0.375)to(m2);
		\draw[->,bend right] (-0.75,-0.375)to(m1);
		\draw[dashed,very thick] (0)--(3);
		\draw[dashed,very thick] (1)--(2);
		\draw (0,-1.8) node {$S_u(p)$};
		\draw (2,-1.8) node {$S_v(p)$};
		\end{tikzpicture}
	};
\end{tikzpicture}
	\caption{Illustration of Lemma~\ref{clm:P_{uv}}. We have $f(pm(x))=\overline{px}$ and $f(pm(y))=\overline{py}$. Each edge in $S_u(p)$ and $S_v(p)$ induces an $f$-free non-edge between the two sets ($f$-free non-edges are depicted in bold).}
	\label{fig-PuvEstRelou}
\end{figure}

\subsection{Preliminaries for the case $P_{uv}=\emptyset$: the $f$-orientation and related lemmas}\label{sec:lemmas-Puvempty}

In this section, we gather some lemmas about the structure of $G$ and $f$ when $P_{uv}=\emptyset$. They will be useful to our proofs but we feel that they could perhaps be used again. So we assume from here on in this in this section that $G$ is a D2C graph with $P_{uv}= \emptyset$. Observe that, by Lemma \ref{clm:PuvEmptySuvEmpty}, $S_{uv} = \emptyset$.

We will use $f$ to define an orientation, called \emph{$f$-orientation}, of the edges induced by $S_u$ and by $S_v$, as follows. Let $ab$ be an edge within $S_u$ or within $S_v$ with $f(ab)=\overline{bc}$. Then, we orient $a$ towards $b$ and we denote the resulting arc by $\overrightarrow{ab}$. This construction is shown in Figure~\ref{fig-orientationF}. Since $f$ is injective (Lemma~\ref{clm:assignIsInj}), each edge of $S_u$ and $S_v$ receives exactly one orientation. From now on, all arcs considered are those of this $f$-orientation. We denote by $N^+(x)$, $N^+[x]$, $N^-(x)$ and $N^-[x]$ the out-neighbourhood, closed out-neighbourhood, in-neighbourhood, and closed in-neighbourhood of vertex $x$ with respect to the $f$-orientation, while $N(x)$ and $N[x]$ continue to denote the neighbourhood and closed neighbourhood of $x$ in $G$.

\begin{figure}[ht!]
	\centering
	\begin{tikzpicture}
	\node[jolinoeud](y) at (0,0) {};
	\node[jolinoeud](z) at (0,1) {};
	\node[jolinoeud](x) at (2,1) {};
	\draw (0,0.5) ellipse (0.5 and 1);
	\draw (2,0.5) ellipse (0.5 and 1);
	\draw (0,-0.3) node {$b$};
	\draw (0,1.3) node {$a$};
	\draw (2,1.3) node {$c$};
	\draw[dashed] (x)--node[midway,below](0){}(y);
	\draw (x)--(z);
	\draw[<-] (y)to(z);
	\draw (-0.8,0.5) node {$S_u$};
	\draw (2.8,0.5) node {$S_v$};
	\draw[->,bend left] (0,0.5)to(0);
	\end{tikzpicture}
	\caption{The $f$-orientation is constructed from the function $f$: we orient all edges within $S_u$ and $S_v$, and an edge $ab$ is oriented from $a$ to $b$ if $f(ab) = \overline{bc}$.}
	\label{fig-orientationF}
\end{figure}
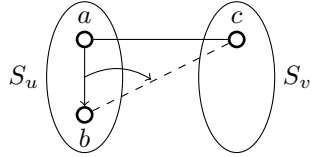

We will now study the properties of the $f$-orientation. The first important lemma is the following (see Figure~\ref{fig-voisinageDroite} for an illustration).

\begin{lemma}\label{clm:arc}
Let $x,y \in S_u$ (resp. $S_v$) be two vertices such that $\overrightarrow{xy}$ is an arc of the $f$-orientation. If neither $x$ nor $y$ is incident with an $f$-free non-edge, then there exists a vertex $t \in S_v$ (resp. $S_u$) such that $N(x)\cap S_v=(N(y)\cap S_v)\cup\{t\}$ (resp. $N(x)\cap S_u=(N(y)\cap S_u)\cup\{t\}$).
\end{lemma}
\begin{proof}
Assume without loss of generality that $x,y \in S_u$.
First, if there exists a vertex $z$ of $N(y)\cap S_v$ that is not adjacent to $x$, then $z$ and $x$ have $y$ as a common neighbour and thus the non-edge $\overline{xz}$ could only be assigned by $f$ to $xy$, contradicting the $f$-orientation of $xy$. Thus, we have $N(y)\cap S_v\subseteq N(x)\cap S_v$. Now, by the $f$-orientation of $xy$, there exists a vertex $t$ of $S_v$ with $f(xy)=\overline{ty}$. Assume now, for a contradiction, that there exists another vertex $t'$ in $S_v$ that is adjacent to $x$ but not to $y$. Then, $t'$ and $y$ have $x$ as a common neighbour, so the non-edge $\overline{t'y}$ can only be assigned by $f$ to $xy$. This is a contradiction and proves the claim.
\end{proof}

\begin{figure}[ht!]
	\centering
	\begin{tikzpicture}
	\node[jolinoeud] (x) at (0,1) {};
	\node[jolinoeud] (y) at (0,0) {};
	\node[jolinoeud] (t) at (2,1) {};
	\draw[->] (x)to(y);
	\draw (x) node[left] {$x$};
	\draw (y) node[left] {$y$};
	\draw (t) node[right] {$t$};
	\draw (0,0.5) ellipse (0.75 and 1.25);
	\draw (2,0.5) ellipse (0.75 and 1.25);
	\draw (0,-1) node {$S_u$};
	\draw (2,-1) node {$S_v$};
	\draw (2,0) ellipse (0.25 and 0.65);
	\node[jolinoeud] (z1) at (2,0.4) {};
	\node[jolinoeud] (z2) at (2,-0.4) {};
	\draw (2,0.15) node[scale=0.5] {$\bullet$};
	\draw (2,0) node[scale=0.5] {$\bullet$};
	\draw (2,-0.15) node[scale=0.5] {$\bullet$};
	\draw (y)--(z1);
	\draw (y)--(z2);
	\draw (x)--(z1);
	\draw (x)--(z2);
	\draw (x)--(t);
	\draw[dashed] (y)--(t);
        \draw[->,bend left] (0,0.5)to(0.5,0.28);
	\end{tikzpicture}
	\caption[voisinageDroite]{Illustration of Lemma~\ref{clm:arc}: if $x$ and $y$ are two vertices in $S_u$ (resp. $S_v$) such that $\overrightarrow{xy}$ is an arc of the $f$-orientation and neither $x$ nor $y$ is incident with an $f$-free non-edge, then the neighbourhood of $x$ in $S_v$ (resp. $S_u$) is exactly the neighbourhood of $y$ in $S_v$ (resp. $S_u$) plus one vertex.}
	\label{fig-voisinageDroite}
\end{figure}
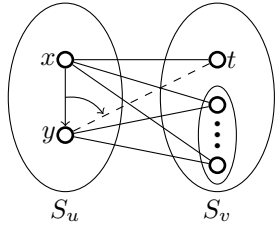

The next lemma states that directed cycles in $G$ yield many $f$-free non-edges.

\begin{lemma}\label{clm:directed-cycle}
Let $\overrightarrow{C}$ be a cycle of $G$ that is directed with respect to the $f$-orientation. Then, there are at least $|\overrightarrow{C}|$ $f$-free non-edges incident with the vertices of $\overrightarrow{C}$.
\end{lemma}
\begin{proof}
Let $\overrightarrow{C}=x_0,x_1,\ldots,x_{k-1}$ be a cycle of $G$ that is directed with respect to the $f$-orientation (with the arc $\overrightarrow{x_ix_{i+1}}$ for each $i$ in $\{ 0,\ldots,k-1 \}$). Without loss of generality, $\overrightarrow{C}$ is in $X$.
In this proof, we consider the addition modulo $k$.

By definition of the $f$-orientation, for all $i \in \{ 0,\ldots,k-1 \}$, there exists a vertex $y_i \in Y$ such that $f(x_ix_{i+1})=\overline{x_{i+1}y_i}$. Note that we may have $y_i=y_j$ for some $i\neq j$.

Let $i \in \{ 0,\ldots,k-1 \}$, and let $x_j$ be the first predecessor of $x_i$ in the cyclic order of $\overrightarrow{C}$ such that $x_jy_i \not\in E$. This clearly happens at some point since $x_{i+2}y_i \notin E$. Note that we have $j \neq i+1$, since otherwise $x_{i+1}$ and $y$ would have two common neighbours: $x_i$ and $x_{i+2}$, a contradiction since $f(x_ix_{i+1})=\overline{x_{i+1}y_i}$ implies that $x_i$ is the unique common neighbour of $y_i$ and $x_{i+1}$.

We now prove that $\overline{x_jy_i}$ is $f$-free. Assume by contradiction that it is not $f$-free. Then, there exists an edge $e$ such that $f(e)=\overline{x_jy_i}$. However, since $x_j$ and $y_i$ have a common neighbour which is $x_{j+1}$, we necessarily have $e=x_jx_{j+1}$. But by definition of $\overrightarrow{C}$, we already have the vertex $y_j \in Y$ such that $f(x_jx_{j+1})=\overline{x_{j+1}y_j}$, and thus $y_i\neq y_j$. This is a contradiction, which implies that $\overline{x_jy_i}$ is $f$-free. This is illustrated in Figure~\ref{fig-directedCycle}.

Finally, we prove that any two $f$-free non-edges found with this method are distinct. Assume by contradiction that there are two vertices $x_{i_1}$, $x_{i_2}$ in $\overrightarrow{C}$ with $i_1<i_2$ which lead to the same $f$-free non-edge $\overline{x_jy}$ (with $y=y_{i_1}=y_{i_2}$). Then, this means that $x_j$ is the first predecessor of $x_{i_2}$ such that $x_jy \not\in E$. In particular, since $i_1<i_2$, this implies that $x_{i_1+1}y \in E$, which contradicts the fact that $f(x_{i_1}x_{i_1+1})=\overline{x_{i_1+1}y}$.

Thus, there are at least $k$ $f$-free non-edges incident with the vertices of $\overrightarrow{C}$.
\end{proof}

\begin{figure}[ht!]
	\centering
	\begin{tikzpicture}[scale=0.75]
	\node[jolinoeud] (xi) at (0,1) {};
	\node[jolinoeud] (xi1) at (0,0) {};
	\node[jolinoeud] (xj) at (0,4) {};
	\node[jolinoeud] (xj1) at (0,3) {};
	\draw[->] (xi)to(xi1);
	\draw[->] (xj)to(xj1);
	\draw[->,dashed] (0,4.5)to(xj);
	\draw[->,dashed] (0,1.5)to(xi);
	\draw[dashed] (xj1)to(0,2.5);
	\draw[dashed] (xi1)to(0,-0.5);
	\draw (xi) node[left] {$x_i$};
	\draw (xi1) node[left] {$x_{i+1}$};
	\draw (xj) node[left] {$x_j$};
	\draw (xj1) node[left] {$x_{j+1}$};
	\draw (0,1.75) node[scale=0.75] {$\bullet$};
	\draw (0,2) node[scale=0.75] {$\bullet$};
	\draw (0,2.25) node[scale=0.75] {$\bullet$};
	\node[jolinoeud] (yi) at (4,1) {};
	\node[jolinoeud] (yj) at (4,4) {};
	\draw (yi) node[right] {$y_i$};
	\draw (yj) node[right] {$y_j$};
	\draw (xi)to(yi);
	\draw[dashed] (xi1)to(yi);
	\draw (xj1)to(yi);
	\draw[dashed,very thick] (xj)to(yi);
	\draw (xj)to(yj);
	\draw[dashed] (xj1)to(yj);
	\draw (0,2) ellipse (1.5 and 3.5);
	\draw (4,2) ellipse (0.75 and 3.5);
	\draw (0,-1.75) node {$S_u$};
	\draw (4,-1.75) node {$S_v$};
	\end{tikzpicture}
	\caption{Illustration of the proof of Lemma~\ref{clm:directed-cycle}: $x_j$ is the first predecessor of $x_i$ in the cycle such that $x_jy_i \notin E$. The non-edge $\overline{x_jy_i}$ is then $f$-free.}
	\label{fig-directedCycle}
\end{figure}
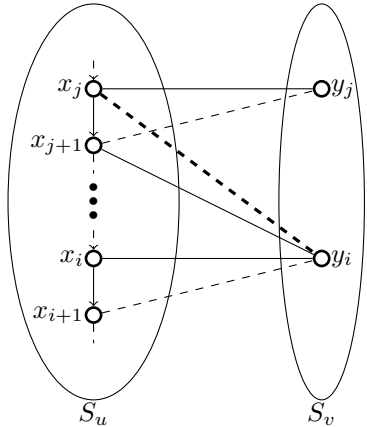

We now show that transitive triangles also induce $f$-free non-edges.

\begin{lemma}\label{clm:triangle}
Let $x,y,z\in S_u$ (resp. $S_v$) be three pairwise adjacent vertices such that $\overrightarrow{xy}$,  $\overrightarrow{xz}$ and  $\overrightarrow{yz}$ are oriented edges. Then there is an $f$-free non-edge incident with $x$.
\end{lemma}
\begin{proof}
Without loss of generality, assume that $x,y,z \in S_u$.
Assume by contradiction that there is no $f$-free non-edge incident with $x$. By Lemma~\ref{clm:everySuHasNeighbInSv}(c), $x$ has a neighbour $w_1$ in $S_v$, and by definition of the $f$-orientation, we can assume that $f(xy)=\overline{yw_1}$. Similarly, there is a vertex $w_2$ of $S_v$ adjacent to $y$ with $f(yz)=\overline{zw_2}$ (clearly $w_1\neq w_2$). Then neither $z$ nor $w_2$ is adjacent to $w_1$ since $x$ is the unique common neighbour of $y$ and $w_1$. Now similarly, $x$ and $w_2$ cannot be adjacent, the only common neighbour of $z$ and $w_2$ is $y$. But then the non-edge $\overline{xw_2}$ is assigned to $xy$ by $f$, which contradicts the injectivity of $f$. This is illustrated in Figure~\ref{fig-triangleInSu}.
\end{proof}

\begin{figure}[ht!]
	\centering
	\begin{tikzpicture}
	\node[jolinoeud] (x) at (0,2) {};
	\node[jolinoeud] (y) at (0,1) {};
	\node[jolinoeud] (z) at (0,0) {};
	\node[jolinoeud] (w1) at (2,2) {};
	\node[jolinoeud] (w2) at (2,1) {};
	\draw (x) node[left] {$x$};
	\draw (y) node[left] {$y$};
	\draw (z) node[below left] {$z$};
	\draw (w1) node[right] {$w_1$};
	\draw (w2) node[right] {$w_2$};
	\draw[->] (x)to(y);
	\draw[->] (y)to(z);
	\draw[->,bend right=45] (x)to(z);
	\draw (x)to(w1);
	\draw (y)to(w2);
	\draw[dashed,very thick] (x)to(w2);
	\draw[dashed] (y)to(w1);
	\draw[dashed] (z)to(w1);
	\draw[dashed] (z)to(w2);
	\draw (0,1) ellipse (0.75 and 1.5);
	\draw (2,1) ellipse (0.75 and 1.5);
	\draw (0,-0.75) node {$S_u$};
	\draw (2,-0.75) node {$S_v$};
	\end{tikzpicture}
	\caption{Illustration of the proof of Lemma~\ref{clm:triangle}: there is an $f$-free non-edge incident with the ancestor in a transitive triangle.}
	\label{fig-triangleInSu}
\end{figure}
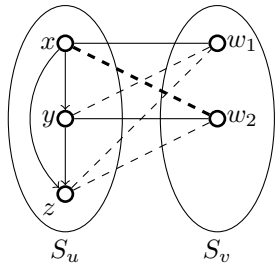

The next two lemmas state that each source and each sink of $G$ has an $f$-free non-edge in its closed neighbourhood. Recall that we do not consider isolated vertices as sources or sinks.

\begin{lemma}\label{clm:sink}
  Let $x$ be a sink of the $f$-orientation. Then, there is at least one $f$-free non-edge incident with the vertices of $N^-[x]$.
\end{lemma}
\begin{proof}
Let $x \in S_u$ (without loss of generality) be a sink of the $f$-orientation, and let $a_1,\ldots,a_k$ be the in-neighbours of $x$ (recall that they all belong to $S_u$). 

By Lemma~\ref{clm:everySuHasNeighbInSv}(c), $x$ has a neighbour $y\in S_v$.
Let $a_i$ be an in-neighbour of $x$. If $a_i$ is not adjacent to $y$, then the non-edge $\overline{a_iy}$ is $f$-free, since otherwise we should have $f(xa_i)=\overline{a_iy}$ and the arc $\overrightarrow{ta_i}$, a contradiction since $x$ is a sink. Thus, from now on we may assume that all in-neighbours of $x$ are adjacent to $y$, for otherwise the statement of the claim holds.

Like every edge, the edge $xy$ is critical. It cannot be critical for the pair $\{x,y\}$, since (by the previous paragraph) these two vertices have all vertices in $N^-(x)$ as common neighbours. It also cannot be critical for a pair consisting of $y$ and a neighbour $z$ of $x$: if $z \in S_u$ then $z$ is an $a_i$; and if $z \in S_v \cup \{u\}$ then $v \in N(y) \cap N(z)$. Hence, $xy$ must be critical for a pair $\{x,t\}$ with $t$ a neighbour of $y$. Since $N(x) \cap N(t) = \{y\}$, we necessarily have $t \in S_v$, and $ta_i \notin E$ for every $a_i$. However, among all the non-edges $\overline{a_it}$ and the non-edge $\overline{xt}$, all but one are $f$-free: their only possible preimage by the function $f$ if $yt$. This is depicted in Figure~\ref{fig-sinkClaim}. Since at least one $a_i$ exists, there is at least one $f$-free non-edge incident with $N^-[x]$, and the claim follows.
\end{proof}

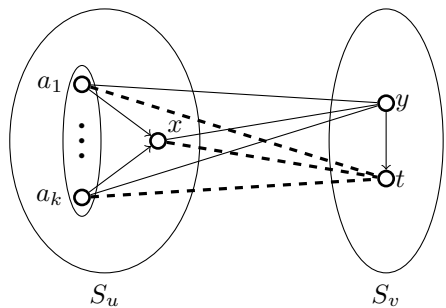
\begin{figure}[ht!]
	\centering
	\begin{tikzpicture}	
	\node[jolinoeud] (x) at (0,0) {};
	\node[jolinoeud] (a1) at (-1,0.75) {};
	\node[jolinoeud] (a2) at (-1,-0.75) {};
	\node[jolinoeud] (y) at (3,0.5) {};
	\node[jolinoeud] (t) at (3,-0.5) {};
	\draw[->] (a1)to(x);
	\draw[->] (a2)to(x);
	\draw[->] (y)to(t);
	\draw (x)to(y);
	\draw (a1)to(y);
	\draw (a2)to(y);
	\draw[dashed,very thick] (x)to(t);
	\draw[dashed,very thick] (a1)to(t);
	\draw[dashed,very thick] (a2)to(t);
	\draw (x) node[above right] {$x$};
	\draw (a1) node[left=0.12] {$a_1$};
	\draw (a2) node[left=0.12] {$a_k$};
	\draw (y) node[right] {$y$};
	\draw (t) node[right] {$t$};
	\draw (-1,0.2) node[scale=0.5] {$\bullet$};
	\draw (-1,0) node[scale=0.5] {$\bullet$};
	\draw (-1,-0.2) node[scale=0.5] {$\bullet$};
	\draw (-1,0) ellipse (0.25 and 1);
	\draw (-0.7,0) ellipse (1.25 and 1.75);
	\draw (3,0) ellipse (0.75 and 1.75);
	\draw (-0.7,-2.05) node {$S_u$};
	\draw (3,-2.05) node {$S_v$};
	\end{tikzpicture}
	\caption{If $x$ is a sink, then its closed in-neighbourhood is incident with $f$-free non-edges: among all the bolded non-edges, only one can have a preimage by the function $f$, this preimage being $yt$.}
	\label{fig-sinkClaim}
\end{figure}

\begin{lemma}\label{clm:source}
	Let $x$ be a source of the $f$-orientation. Then, there is at least one $f$-free non-edge incident with the vertices of $N^+[x]$.
\end{lemma}
\begin{proof}
	Let $x \in S_u$ (without loss of generality) be a source of the $f$-orientation, and let $b_1,\ldots,b_k$ be the out-neighbours of $x$ (recall that they all belong to $S_u$). Assume by contradiction that no $f$-free non-edge is incident with $N^+[x]$.
	
	By Lemma~\ref{clm:everySuHasNonNeighbInSv}(d), $x$ has a non-neighbour $y \in S_v$. Lemma~\ref{clm:arc} ensures that no $b_i$ is adjacent to $y$. The non-edge $\overline{xy}$ has a preimage by the function $f$, and this preimage is necessarily an arc $\overrightarrow{zy}$ with $z \in S_v$, since the $b_i$ are the only neighbours of $x$ in $S_u$ and none of those are adjacent to $y$.
	
	Let $i \in \{1,\ldots,k\}$. If $b_iz \in E$, then $z \in N(y) \cap N(b_i)$, and since no non-edge incident with a $b_i$ is $f$-free we necessarily have $f(zy)=\overline{b_iy}$, a contradiction to the fact that $f(zy)=\overline{xy}$. Thus, $\overline{b_iz}$ is a non-edge incident with $N^+[x]$, thus it has a preimage by the function $f$. This preimage is necessarily the edge $xb_i$.
	
	Furthermore, by Lemma~\ref{clm:everySuHasNeighbInSv}(c), $b_i$ has a neighbour $w \in S_v$. Lemma~\ref{clm:arc} ensures that $xw \in E$. Now, the edge $xw$ is critical. It cannot be critical for the pair $\{x,w\}$ since those vertices share $b_i$ as a common neighbour. It cannot be critical for a pair $\{x,t\}$ with $t \in S_v$ since by definition of critical edges we would have $b_it \notin E$, and one of the two non-edges $\overline{xt}$, $\overline{b_it}$ would be $f$-free, a contradiction. Thus, there is a $b_j \in N^+(x)$ such that $f(xb_j)=\overline{b_jw}$. However, Lemma~\ref{clm:arc} ensures that $N(b_j) \cap S_v = (N(x) \cap S_v) \setminus \{z\}$, which implies that $b_jw \in E$. This is a contradiction, which proves the claim.
\end{proof}

\subsection{Proof of Conjecture~\ref{conj:MS} for graphs with a dominating edge}\label{sec:MSproof}

We are now ready to re-prove Conjecture~\ref{conj:MS} for graphs with a dominating edge. The proof only uses the content of Sections~\ref{sec:prelim} and \ref{sec:lemmas-Puvnonempty} and Lemma~\ref{clm:arc}. More specifically, we will prove the following lemma. Observe that, to prove Conjecture~\ref{conj:MS}, we may assume that $G$ is non-bipartite (for otherwise the conjecture is true for $G$ by Lemma~\ref{clm:MSBound}). We use the $f$-orientation that was defined in Section~\ref{sec:lemmas-Puvempty}.

\begin{lemma}\label{lem:sol_of_conj1}
If $G$ is a non-bipartite D2C graph with a dominating edge, then $\free(f)\geq 1$.
\end{lemma}

\begin{proof}
We assume by contradiction that $\free(f)=0$. We distinguish two cases.

\paragraph{Case 1: \boldmath{$P_{uv}$} is nonempty.} Note that $S_u\neq\emptyset$, for otherwise $G$ must be a star.
Since $\free(f)=0$, by the definition of $S_v(p)$ (see Lemma~\ref{eq:Puv}) we have $S_v=S_v(p)$ for each vertex $p$ of $P_{uv}$, and $|S_u|\geq |S_v|$ by Lemma~\ref{eq:Puv}(d). By Lemma~\ref{clm:MSBound}, if $||X|-|Y||\geq 3$, $G$ has at most $\left\lfloor(n^2-1)/4\right\rfloor-2$ edges, and we are done. Thus, we may assume that $||X|-|Y||\leq 2$. This leaves several possibilities depending on the structure of $G$:

\begin{enumerate}
	\item First, if $S_{uv}$ is nonempty, then $|S_{uv}|=1$, $P_{uv}=\{p\}$, and $|S_u|=|S_v|$. Thus, $S_u=S_u(p)$, $S_v = S_v(p)$ and the only edges between $S_u$ and $S_v$ are of the form $xm(x)$. This implies that $p$ is adjacent to all vertices of $S_u$ and that the vertex in $S_{uv}$ cannot be adjacent to any vertex $x \in S_u$ (otherwise the edge $ux$ would not be critical). It follows that the endpoints of any non-edge in $S_u(p) \times S_v(p)$ have no common neighbour, which is not possible. Thus, there is no such non-edge, which implies by Lemma~\ref{clm:P_{uv}} that $|S_u|=|S_v|=1$; but then the graph is $H_5$ which has exactly $\left\lfloor n^2/4\right\rfloor-1$ edges, a contradiction. Thus, we can assume that $S_{uv}$ is empty.
	\item Now, if $|S_u|=|S_v|$ (that is, $S_u=S_u(p)$ for every $p \in P_{uv}$), then for any $x\in S_u$, the edge $ux$ is not critical, a contradiction.
	\item Thus, there exists exactly one vertex $z$ in $S_u \setminus S_u(p)$ and exactly one vertex $p$ in $P_{uv}$. Since $\free(f)=0$, by Lemma~\ref{clm:P_{uv}}(c) and Lemma~\ref{eq:Puv}(d), $S_u(p)$ and $S_v(p)$ are independent sets. If $|S_u|=2$ (which implies $|S_v|=1$), then by Lemma~\ref{clm:everySuHasNeighbInSv}(c) $z$ is adjacent to the vertex in $S_v$. But then, the edge between $u$ and the vertex of $S_u(p)$ is not critical, a contradiction. Thus, $|S_u(p)|\geq 2$. Let $x,y$ be two distinct vertices of $S_v(p)$. 
	If $zm(x) \notin E$, then the non-edge $\overline{m(x)y}$ has no preimage by $f$, and thus it is $f$-free, a contradiction. Hence, $zm(x) \in E$ and $f(zm(x))=\overline{m(x)y}$, which implies that $z$ is adjacent to every vertex in $S_u(p)$ and in $S_v(p)$. However, the edges between $u$ and $S_u(p)$ are not critical, a contradiction.
\end{enumerate}

This study covers all possible cases, and we always reach a contradiction. This finsihes the proof of Case~$1$.

\paragraph{Case 2: \boldmath{$P_{uv}$} is empty.} Note that $S_{uv}=\emptyset$ by Lemma~\ref{clm:PuvEmptySuvEmpty}(b). Furthermore, there is at least one edge in $S_u$ (without loss of generality) since otherwise the graph is bipartite, a contradiction.
 We prove the following statement:

\begin{claim}
Claim: Let $\overrightarrow{xy}$ be an arc in $S_u$. For every vertex $t \in N(y) \cap S_v$, the edge $yt$ is critical for a pair $\{z,t\}$ with $z$ an out-neighbour of $y$.
\end{claim}
\begin{claim_proof}
	To prove the claim, observe first that the edge $yt$ cannot be critical for the pair $\{y,t\}$, since by Lemma~\ref{clm:arc} (which we can apply since we assume that $\free(f)=0$) we have $N(y) \cap S_v \subset N(x) \cap S_v$, and so $y$ and $t$ have $x$ as a common neighbour.
	
	Assume by contradiction that the edge $yt$ is critical for a pair $\{y,w\}$ with $w \in S_v$ (note that $w$ is an out-neighbour of $t$). We have $xw \notin E$ since otherwise $y$ and $w$ would have two common neighbours, a contradiction. Thus, we have the non-edge $\overline{xw}$. This non-edge has a preimage by the function $f$, but this preimage can only be $tw$ since $t \in N(w) \cap N(x)$. This implies that one of the two non-edges $\overline{xw}$, $\overline{yw}$ is $f$-free, a contradiction.
	
	Thus, the edge $yt$ is critical for a pair $\{z,t\}$, and since $\free(f)=0$ we necessarily have $f(yz)=\overline{zt}$ and so $z$ is an out-neighbour of $y$ by definition of the $f$-orientation. Hence, the claim follows.
	\end{claim_proof}

Now, take a maximal directed path $\overrightarrow{x_1,\ldots,x_k}$ of vertices in $S_u$. By Lemma~\ref{clm:everySuHasNeighbInSv}, $x_k$ has a neighbour $t \in S_v$. By the above claim, 
the edge $x_kt$ is critical for a pair $\{y,t\}$ with $y$ an out-neighbour of $x_k$. We cannot have $y=x_i$ for $i \in \{1,\ldots,k-2\}$ since otherwise we would have a directed cycle, and by Lemma~\ref{clm:directed-cycle} we would have $k-i+1$ $f$-free non-edges incident with vertices in the cycle, a contradiction. Thus, the directed path $\overrightarrow{x_1,\ldots,x_ky}$ is a directed path in $S_u$ with more vertices than $\overrightarrow{x_1,\ldots,x_k}$, a contradiction.

Hence, we have proved that if $\free(f)=0$ then we reach a contradiction, and the statement of the lemma follows.
\end{proof}

Hence, Lemma \ref{lem:sol_of_conj1} confirms Conjecture \ref{conj:MS} for D2C graphs with a dominating edge, i.e. we obtain a new proof of the following theorem.

\begin{theorem}[\cite{HW03,diam3,correction,Wang-arxiv}]
Any D2C graph $G$ of order $n$ with a dominating edge has at most $\lfloor n^2/4\rfloor$ edges, with equality if and only if $G$ is a balanced complete bipartite graph.
\end{theorem}

\subsection{A stronger theorem when $P_{uv}=\emptyset$}\label{sec:strong-theorem}

In this section, we use the lemmas of Section~\ref{sec:lemmas-Puvempty} to prove the following.

\begin{theorem}\label{thm:strong}
  Let $G$ be a non-bipartite D2C graph with a dominating edge $uv$ such that $P_{uv}=\emptyset$, and let $f$ be the associated injective function. Let $\overrightarrow{D}$ be the graph induced by $S_x$ ($x\in\{u,v\}$) and oriented with respect to the $f$-orientation.
  Let $\mathcal C\cup \mathcal S$ be a collection of vertex-disjoint subgraphs of $\overrightarrow{D}$ satisfying the following conditions:
  \begin{enumerate}
  \item $\mathcal C$ consists of directed cycles;
  \item $\mathcal S$ consists of transitive triangles and graphs with a universal vertex that is either a sink or a source in $\overrightarrow{D}$.
  \end{enumerate}
Then, $G$ has at most $\left\lfloor n^2/4\right\rfloor-\sum_{C\in \mathcal C}|C|-|\mathcal S|$ edges.
\end{theorem}
\begin{proof}
  By Lemmas~\ref{clm:directed-cycle}, \ref{clm:triangle}, \ref{clm:sink} and \ref{clm:source}, we have $\free(f)\geq \sum_{C\in \mathcal C}|C|+|\mathcal S|$. Thus, the bound follows from Lemma~\ref{clm:MSBound}.
\end{proof}


Observe that each connected component of $\overrightarrow{D}$ contains either a directed cycle, or a source. Thus, by Lemmas~\ref{clm:directed-cycle} and~\ref{clm:source}, we have $\free(f)\geq c$, where $c$ is the number of nontrivial connected components of $\overrightarrow{D}$. This implies, by Lemma~\ref{clm:MSBound}, that $G$ has at most $\left\lfloor n^2/4\right\rfloor-c$ edges. But in fact, we can prove the following stronger result, which is crucial for the proof of Theorem~\ref{thm:main}.

\begin{theorem}\label{thm:strong2}
  Let $G$ be a non-bipartite D2C graph with a dominating edge $uv$ such that $P_{uv}=\emptyset$, and let $f$ be the associated injective function. Let $\overrightarrow{D}$ be the graph induced by $S_x$ ($x\in\{u,v\}$) and oriented with respect to the $f$-orientation, and let $c_1$ (resp. $c_2$) be the number of nontrivial connected components of diameter~2 (resp. at least~3) in $\overrightarrow{D}$. Then, $G$ has at most $\left\lfloor n^2/4\right\rfloor-c_1-2c_2$ edges.
\end{theorem}
\begin{proof}
Without loss of generality, we assume that $\overrightarrow{D}$ is the oriented graph induced by $S_u$. Note that every nontrivial component in $\overrightarrow{D}$ contains either a directed cycle or a source and a sink. Thus, there is at least one $f$-free non-edge with one endpoint in $C$, which proves that there are at most $\left\lfloor n^2/4 \right\rfloor-c_1-c_2$ edges in $G$. We now assume that $C$ has diameter at least~3 in $\overrightarrow{D}$. We must show that $G$ contains at least two $f$-free non-edges with one endpoint in $C$, which will prove the theorem.

If there is a directed cycle in $C$, we are done by Lemma~\ref{clm:directed-cycle}. Thus, we assume that $\overrightarrow{D}[C]$ is acyclic. This implies that there is at least one source and at least one sink in $C$. Let $\mathcal S$ and $\mathcal T$ be the sets of sources and sinks of $C$, respectively. We assume by contradiction that there is at most one $f$-free non-edge incident with $C$.

We recall that $C$ is acyclic and thus both $\mathcal S$ and $\mathcal T$ are nonempty. By Lemmas~\ref{clm:sink} and~\ref{clm:source}, each source and each sink are at distance at most~$1$ from a vertex $r$ of $C$ incident with an $f$-free non-edge (this implies that there is exactly one $f$-free non-edge incident with $C$). Thus each vertex of $\mathcal S\cup\mathcal T)  \setminus\{r\}$ is adjacent to $r$; if $r\in\mathcal S$ then $\mathcal S=\{r\}$; if $r\in\mathcal T$ then $\mathcal T=\{r\}$. By Lemma~\ref{clm:triangle}, if $r\notin(\mathcal S\cup\mathcal T)$, there is no arc from any source to any sink.

We now prove a more constrained structure on $r$ and $\mathcal{T}$.

\begin{claim}\label{clm:2.1}
Claim 1: Either $r\in\mathcal T$, or $N^-(t) = \{r\}$ for all $t \in \mathcal T$.
\end{claim}
\begin{claim_proof}
By contradiction, assume that $r\notin\mathcal T$ but that some sink $t\in\mathcal T$ has two in-neighbours $r$ and $x$. By Lemma~\ref{clm:triangle}, $r \notin N^+(x)$, otherwise we have two $f$-free non-edges incident with $C$ (in particular, $x\notin\mathcal S$). By definition of the $f$-orientation, there is a vertex $y \in S_v$ such that $f(xt)=\overline{yt}$. Furthermore, by Lemma~\ref{clm:everySuHasNeighbInSv}(c), $t$ has a neighbour in $S_v$, which we will call $p$. We have $py \notin E$ and $xp \in E$ by Lemma~\ref{clm:arc}.

Assume first that $r \in N^-(x)$. Then $ry \notin E$, since $N(t) \cap N(y) = \{x\}$. If $\overline{ry}$ had a preimage by the function $f$, it would be the edge $xr$ and we would have the arc $\overrightarrow{xr}$, a contradiction. Thus, $\overline{ry}$ is the only $f$-free non-edge in $G$. This is depicted in Figure~\ref{fig-caseStudy1}. Now, the edge $tp$ is critical. It is not critical for the pair $\{t,p\}$ since $x \in N(t) \cap N(p)$. It is not critical for a pair $\{p,z\}$ with $z \in S_u$ since otherwise we would have the arc $\overrightarrow{zt}$ (since $t$ is a sink) and thus the non-edge $\overline{zp}$ would be $f$-free, a contradiction. It is not critical for a pair $\{t,z\}$ with $z \in S_v$, since this would imply that $xz \notin E$ and thus one of the two non-edges $\overline{tz}$, $\overline{xz}$ would be $f$-free (their only possible preimage by the function $f$ would be $pz$), a contradiction. Thus, $tp$ is not critical; this contradiction implies that $r \notin N(x)$.

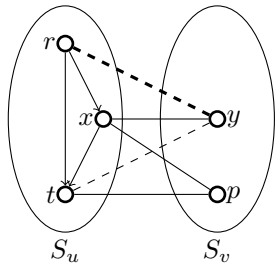
\begin{figure}[ht!]
	\centering
	\begin{tikzpicture}
	\node[jolinoeud] (x) at (0.5,1) {};
	\node[jolinoeud] (r) at (0,2) {};
	\node[jolinoeud] (t) at (0,0) {};
	\node[jolinoeud] (y) at (2,1) {};
	\node[jolinoeud] (p) at (2,0) {};
	\draw[->] (x)to(t);
	\draw[->] (r)to(x);
	\draw[->] (r)to(t);
	\draw[dashed] (t)to(y);
	\draw[dashed,very thick] (r)to(y);
	\draw (x)to(y);
	\draw (x)to(p);
	\draw (t)to(p);
	\draw (x) node[left] {$x$};
	\draw (r) node[left] {$r$};
	\draw (t) node[left] {$t$};
	\draw (p) node[right] {$p$};
	\draw (y) node[right] {$y$};
	\draw (0,1) ellipse (0.75 and 1.5);
	\draw (2,1) ellipse (0.75 and 1.5);
	\draw (0,-0.75) node {$S_u$};
	\draw (2,-0.75) node {$S_v$};
	\end{tikzpicture}
	\caption{If $r \in N^-(x)$, then the non-edge $\overline{ry}$ (depicted in bold), is the only $f$-free non-edge in $G$. By studying all possibilities for the criticality of the edge $tp$, we then reach a contradiction.}
	\label{fig-caseStudy1}
\end{figure}

The edge $xp$ is critical. It is not critical for the pair $\{x,p\}$ since $t \in N(x) \cap N(p)$. It is not critical for a pair $\{x,z\}$ with $z \in S_v$ since otherwise we would have $tz \notin E$ and thus one of the two non-edges $\overline{xz}$, $\overline{tz}$ would be $f$-free, a contradiction. Thus, it is critical for a pair $\{w,p\}$ with $w \in N^+(x)$. In particular, we necessarily have $wy \in E$ (since otherwise the non-edge $\overline{wy}$ would be $f$-free since its only possible preimage by the function would be $xw$ which is already assigned, a contradiction).

Now, let us reexamine the edge $tp$, which is critical. By the same reasoning that the one we held just above, it is critical for a pair $\{p,z\}$ with $z \in S_u$. Since $t$ is a sink, the non-edge $\overline{zp}$ is $f$-free (since otherwise we would have an arc $\overrightarrow{tz}$, a contradiction with the fact that $t$ is a sink), and thus $z=r$ and $\overline{zr}$ is the only $f$-free non-edge in $G$. Note that if $N(t) \cap S_v$ contains more than one vertex, then we can repeat the argument and find other $f$-free non-edges, a contradiction. Thus, $N(t) \cap S_v = \{p\}$. The construction we obtain is depicted in Figure~\ref{fig-caseStudy2}.

\begin{figure}[ht!]
	\centering
	\begin{tikzpicture}
	\node[jolinoeud] (x) at (0,2) {};
	\node[jolinoeud] (w) at (-0.5,1) {};
	\node[jolinoeud] (r) at (0.5,1) {};
	\node[jolinoeud] (t) at (0,0) {};
	\node[jolinoeud] (y) at (2,2) {};
	\node[jolinoeud] (p) at (2,0) {};
	\draw[->] (x)to(t);
	\draw[->] (x)to(w);
	\draw[->] (r)to(t);
	\draw[dashed] (t)to(y);
	\draw[dashed] (w)to(p);
	\draw[dashed,very thick] (r)to(p);
	\draw (x)to(y);
	\draw (w)to(y);
	\draw (x)to(p);
	\draw (t)to(p);
	\draw (x) node[left] {$x$};
	\draw (r) node[left] {$r$};
	\draw (w) node[below] {$w$};
	\draw (t) node[left] {$t$};
	\draw (p) node[right] {$p$};
	\draw (y) node[right] {$y$};
	\draw (0,1) ellipse (0.75 and 1.5);
	\draw (2,1) ellipse (0.75 and 1.5);
	\draw (0,-0.75) node {$S_u$};
	\draw (2,-0.75) node {$S_v$};
	\end{tikzpicture}
	\caption{If $r \notin N(x)$, then the non-edge $\overline{rp}$ (depicted in bold), is the only $f$-free non-edge in $G$. By studying all possibilities for the criticality of the edge $wy$, we then reach a contradiction.}
	\label{fig-caseStudy2}
\end{figure}
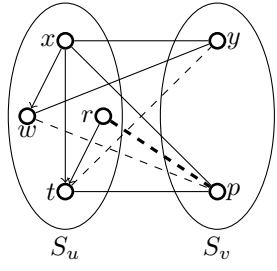

The edge $wy$ is critical. It is not critical for the pair $\{w,y\}$ since $x \in N(w) \cap N(y)$. It is not critical for a pair $\{w,z\}$ with $z \in S_v$ since otherwise we would have $xz \notin E$ and thus one of the two non-edges $\overline{xz}$, $\overline{wz}$ would be $f$-free, a contradiction. Thus, it is critical for a pair $\{y,z\}$ with $z \in S_u$. If $z \in N^-(w)$ then $\overline{zy}$ is $f$-free, which implies that $z=r$, but since $y \neq p$ we have two $f$-free non-edges incident with $C$, a contradiction. Thus, $z \in N^+(w)$ and $zp \notin E$ (since otherwise this would contradict the fact that $f(xw)=\overline{wp}$). By Lemma~\ref{clm:everySuHasNeighbInSv}(c), $z$ has a neighbour $y' \in S_v$. Recall that $N(t) \cap S_v = \{p\}$ and thus $ty' \notin E$. Furthermore, by Lemma~\ref{clm:arc}, $x$ has exactly one neighbour in $S_v$ that is not a neighbour of $t$. This neighbour is $y$, which implies that $xy' \notin E$. By applying the same argument, we also get $wy' \notin E$. But, since we have the arc $\overrightarrow{wz}$, the non-edge $\overline{wy'}$ is $f$-free, a contradiction.

Thus, such an $x$ does not exist, which proves the claim.
\end{claim_proof}

We can now prove that there is exactly one source in $C$.

\begin{claim}\label{clm:one-source}
Claim 2: $|\mathcal S|=1$.
\end{claim}
\begin{claim_proof}
Assume by contradiction that there are two sources $s_1,s_2 \in \mathcal{S}$. By Lemma~\ref{clm:source}, if the vertex $r$ is not in $N^+(s_1) \cap N^+(s_2)$ then there is an $f$-free non-edge incident with one vertex that is not $r$, a contradiction. Thus, $s_1$ and $s_2$ are both in-neighbours of $r$. By definition, there are two distinct vertices $y_1,y_2 \in S_v$ such that $f(s_1r)=\overline{ry_1}$ and $f(s_2r)=\overline{ry_2}$. Furthermore, $s_1y_2$ and $s_2y_1$ cannot be edges, and since $r$ is not a source, these non-edges have to be assigned by $f$. The non-edge $\overline{s_1y_2}$ cannot be assigned to an edge $s_1x$ with $x$ a neighbour of $s_1$ in $S_u$ since $s_1$ is a source and $x$ would be an in-neighbour of $s_1$.

First, assume that we have two distinct vertices $z_1,z_2 \in S_v$ such that $f(y_1z_1)=\overline{s_2y_1}$ and $f(y_2z_2)=\overline{s_1y_2}$. The vertices $z_1$ and $z_2$ cannot be adjacent to $r$ since they are neighbours with respectively $y_1$ and $y_2$. But since $s_1$ is adjacent to $z_2$ and $s_2$ to $z_1$, both non-edges $\overline{rz_1},\overline{rz_2}$ are $f$-free, a contradiction.

Thus, there is a vertex $z \in S_v$ such that $f(y_1z)=\overline{s_2y_1}$ and $f(y_2z)=\overline{s_1y_2}$. By the same argument, $z$ cannot be adjacent to $r$, and the non-edge $\overline{rz}$ is $f$-free. This construction is depicted in Figure~\ref{fig-caseStudy3}.

\begin{figure}[ht!]
	\centering
	\begin{tikzpicture}
	\node[jolinoeud] (s1) at (0,2) {};
	\node[jolinoeud] (s2) at (0,1) {};
	\node[jolinoeud] (r) at (0,0) {};
	\node[jolinoeud] (y1) at (2,2) {};
	\node[jolinoeud] (y2) at (2,1) {};
	\node[jolinoeud] (z) at (2,0) {};
	\draw (s1) node[left] {$s_1$};
	\draw (s2) node[left] {$s_2$};
	\draw (r) node[left] {$r$};
	\draw (y1) node[right] {$y_1$};
	\draw (y2) node[right] {$y_2$};
	\draw (z) node[right] {$z$};
	\draw[->,bend right=45] (s1)to(r);
	\draw[->] (s2)to(r);
	\draw[->,bend right=45] (z)to(y1);
	\draw[->] (z)to(y2);
	\draw (s1)to(y1);
	\draw (s2)to(y2);
	\draw (s2)to(z);
	\draw (s2)to(z);
	\draw[dashed] (s1)to(y2);
	\draw[dashed] (s2)to(y1);
	\draw[dashed] (r)to(y2);
	\draw[dashed] (r)to(y1);
	\draw[dashed,very thick] (r)to(z);
	\draw (0,1) ellipse (0.75 and 1.5);
	\draw (2,1) ellipse (0.75 and 1.5);
	\draw (0,-0.75) node {$S_u$};
	\draw (2,-0.75) node {$S_v$};
	\end{tikzpicture}
	\caption{The vertices $s_1$ and $s_2$ are sources, and $\overline{rz}$ is the only $f$-free non-edge in $G$.}
	\label{fig-caseStudy3}
\end{figure}
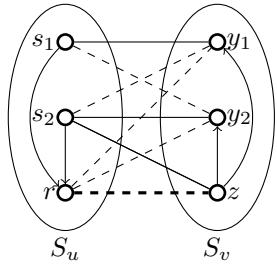

Now, by Lemma~\ref{clm:everySuHasNeighbInSv}(c), $r$ has a neighbour $y_3$ in $S_v$. The vertex $y_3$ is adjacent to $s_1$ and $s_2$ since otherwise the non-edges would be $f$-free. Now, the edge $ry_3$ is critical (note that it is not critical for the pair $\{r,y_3\}$ since they share $s_1$ and $s_2$ as common neighbours). Assume it is critical for a pair $\{r,y_4\}$ with $y_4 \in S_v$. The vertex $y_4$ cannot be $z$ since $r$ and $z$ already have two common neighbours, and it cannot be $y_1$ or $y_2$ since this would contradict $f(s_1r)=\overline{ry_1}$ and $f(s_2r)=\overline{ry_2}$.
However, this implies that $y_4$ is not adjacent to both $s_1$ and $s_2$, and the non-edges $\overline{s_1y_4}$ and $\overline{s_2y_4}$ are $f$-free, a contradiction.

Thus, the edge $ry_3$ is critical for a pair $\{t,y_3\}$ with $t \in S_u$. Moreover, $t \in \mathcal{T}$ since $t$ is an out-neighbour of $r$. By Claim~1, 
$r$ is the unique in-neighbour of $t$. Now, by Lemma~\ref{clm:everySuHasNeighbInSv}(c), $t$ has a neighbour in $S_v$. It cannot be $y_1$ or $y_2$ since $s_1$ and $s_2$ are the unique common neighbours of those and $r$. It cannot be $z$ since this would imply $f(y_1z)=\overline{y_1t}$ and $f(y_2z)=\overline{y_2t}$, a contradiction with $f(y_1z)=\overline{s_2y_1}$ and $f(y_2z)=\overline{s_1y_2}$.

Now, the non-edge $\overline{ty_1}$ has a preimage by the function $f$. This preimage cannot be an edge $xt$ with $x \in S_u$ since $t$ only has $r$ as an in-neighbour, and $ry_1 \notin E$. So there exists a vertex $y_4 \in S_v$ such that $f(y_1y_4)=\overline{ty_1}$. By our previous discussion, $y_4 \notin \{y_2,y_3,z\}$. But now $r$ and $y_4$ cannot be adjacent since $r$ and $y_1$ have $s_1$ as unique common neighbour, which implies that the non-edge $\overline{ry_4}$ is $f$-free, a contradiction that completes this proof. The construction is depicted in Figure~\ref{fig-caseStudy4}.
\end{claim_proof}

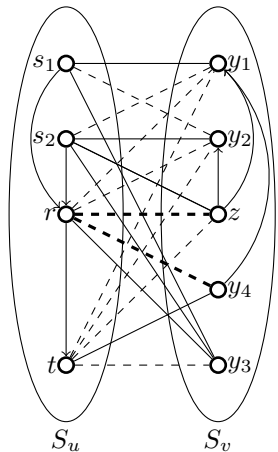
\begin{figure}[ht!]
	\centering
	\begin{tikzpicture}
	\node[jolinoeud] (s1) at (0,2) {};
	\node[jolinoeud] (s2) at (0,1) {};
	\node[jolinoeud] (r) at (0,0) {};
	\node[jolinoeud] (t) at (0,-2) {};
	\node[jolinoeud] (y1) at (2,2) {};
	\node[jolinoeud] (y2) at (2,1) {};
	\node[jolinoeud] (z) at (2,0) {};
	\node[jolinoeud] (y3) at (2,-2) {};
	\node[jolinoeud] (y4) at (2,-1) {};
	\draw (s1) node[left] {$s_1$};
	\draw (s2) node[left] {$s_2$};
	\draw (r) node[left] {$r$};
	\draw (t) node[left] {$t$};
	\draw (y1) node[right] {$y_1$};
	\draw (y2) node[right] {$y_2$};
	\draw (z) node[right] {$z$};
	\draw (y3) node[right] {$y_3$};
	\draw (y4) node[right] {$y_4$};
	\draw[->,bend right=45] (s1)to(r);
	\draw[->] (s2)to(r);
	\draw[->] (r)to(t);
	\draw[->,bend right=45] (z)to(y1);
	\draw[->,bend right=45] (y4)to(y1);
	\draw[->] (z)to(y2);
	\draw (s1)to(y1);
	\draw (s2)to(y2);
	\draw (s2)to(z);
	\draw (s2)to(z);
	\draw (s1)to(y3);
	\draw (s2)to(y3);
	\draw (r)to(y3);
	\draw (t)to(y4);
	\draw[dashed] (s1)to(y2);
	\draw[dashed] (s2)to(y1);
	\draw[dashed] (r)to(y2);
	\draw[dashed] (r)to(y1);
	\draw[dashed] (t)to(y1);
	\draw[dashed] (t)to(y2);
	\draw[dashed] (t)to(z);
	\draw[dashed] (t)to(y3);
	\draw[dashed,very thick] (r)to(z);
	\draw[dashed,very thick] (r)to(y4);
	\draw (0,0) ellipse (0.75 and 2.75);
	\draw (2,0) ellipse (0.75 and 2.75);
	\draw (0,-3) node {$S_u$};
	\draw (2,-3) node {$S_v$};
	\end{tikzpicture}
	\caption{If there are two sources in $C$, then $G$ has at least two $f$-free non-edges incident with $C$. Here, these two $f$-free non-edges are $\overline{rz}$ and $\overline{ry_4}$.}
	\label{fig-caseStudy4}
\end{figure}

Now, by Claim~2, 
there is a unique source $s$ in $C$. By Claim~1, 
 $s\neq r$, and furthermore, since $C$ has diameter at least~$3$, $s$ has an out-neighbour $x$ that is not $r$.

By Lemma~\ref{clm:everySuHasNonNeighbInSv}(d), $s$ has a non-neighbour $y_1$ in $S_v$. Since $s\neq r$, the non-edge $\overline{sy_1}$ has a preimage by the function $f$. This preimage cannot be an edge $sz$ with $z$ a neighbour of $s$ in $S_u$: indeed, since $s$ is a source, $z$ would be an out-neighbour of $s$, but then $\overline{sy_1}$ would be $f$-free. Thus, we have $\overline{sy_1}=f(y_1y_2)$, for some in-neighbour of $y_1$ in $S_v$. Note that, by Lemma~\ref{clm:arc}, we have $N(x) \cap S_v = (N(s) \cap S_v) \setminus \{y_2\}$.
In particular, $xy_1,xy_2 \notin E$.

Now, the non-edge $\overline{xy_1}$ has a preimage by the function $f$. This preimage cannot be an edge $y_1z$ with $z \in S_v$, since otherwise we would have $sz \in E$ and thus $\overline{sy_1}$ would be $f$-free, a contradiction. Thus, there exists $z \in S_u$ an in-neighbour of $x$ such that $f(zx)=\overline{xy_1}$. In particular, we have $z\neq r$ (since otherwise $s$, $r$ and $x$ induce a transitive triangle, and Lemma~\ref{clm:triangle} implies that an $f$-free non-edge is incident with $s$, a contradiction).
This construction is depicted in Figure~\ref{fig-caseStudy5}.

\begin{figure}[ht!]
	\centering
	\begin{tikzpicture}
	\draw (1,1) ellipse (1.5 and 1.75);
	\draw (5,1) ellipse (1 and 1.5);
	\draw (1,-1) node {$S_u$};
	\draw (5,-1) node {$S_v$};
	\node[jolinoeud] (s) at (1,2) {};
	\node[jolinoeud] (r) at (0,1) {};
	\draw[->] (s)to(r);
	\draw (s) node[above=0.1] {$s$};
	\draw (r) node[left=0.1] {$r$};
	\node[jolinoeud] (x) at (2,1) {};
	\draw (x) node[left=0.1] {$x$};
	\draw[->] (s)to(x);
	\node[jolinoeud] (y1) at (5,0) {};
	\draw (y1) node[right=0.1] {$y_1$};
	\draw[dashed] (s)to node[very near end](m1){} (y1);
	\draw[dashed] (x)to node[very near start](m3){} (y1);
	\node[jolinoeud] (y2) at (5,2) {};
	\draw (y2) node[right=0.1] {$y_2$};
	\draw[->] (y2)to node[near end](m2){} (y1);
	\draw (s)to(y2);
	\draw[dashed] (x)to(y2);
	\node[jolinoeud] (z) at (1,0) {};
	\draw (z) node[below=0.1] {$z$};
	\draw[->] (z)to node[near end](m4){} (x);
	\draw (z)to(y1);
	\draw[->,decorate,decoration={snake,amplitude=.7mm}] (z)to(r);
	\draw[->,decorate,decoration={snake,amplitude=.7mm}] (s)to(z);
	\end{tikzpicture}
	\caption{The structure of $C$ is very constrained: there is a unique source $s$, which has an out-neighbour $x$ distinct from $r$. The snake-like arcs represent a directed path of any length: $z$ is an in-neighbour of $r$ and an out-neighbour of $s$ due to the fact that $C$ is acyclic.}
	\label{fig-caseStudy5}
\end{figure}
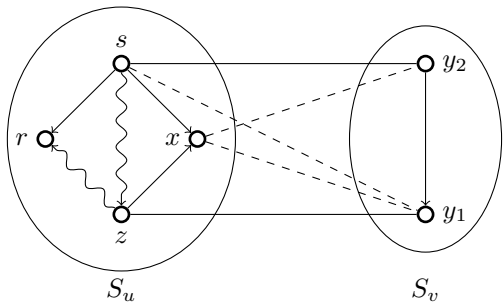

However, Lemma~\ref{clm:arc} implies that no successor of $s$ and predecessor of $r$ can be adjacent to $y_1$.\footnote{A \emph{predecessor} of a vertex $a$ is a vertex $b$ such that there is a directed path from $b$ to $a$. Similarly, a \emph{successor} of $a$ is a vertex $b$ such that there is a directed path from $a$ to $b$.} But since $s$ is the only source in $C$ (by Claim~2) 
and $r$ is the only in-neighbour of all sinks in $C$ (by Claim~1, 
$z$ is a successor of $s$ and a predecessor of $r$. Since $zy_1 \in E$, we reach a contradiction.

Thus, there are at least two $f$-free non-edges incident with vertices of $C$, which completes the proof.
\end{proof}

\subsection{Proof of Theorem~\ref{thm:main}}\label{sec:mainproof}

We are now ready to prove Theorem~\ref{thm:main}, which we recall here.\\

\noindent
{\bf Theorem 4.} {\it 
Let $G$ be a non-bipartite D2C graph with $n$ vertices having a dominating edge. If $G$ is not $H_5$, then $G$ has at most $\left\lfloor n^2/4\right\rfloor-2$ edges.}

\begin{proof}
Consider the function $f$ defined previously. We have seen in Section~\ref{sec:MSproof} that $\free(f)\geq 1$. We need to prove that $\free(f)\geq 2$. Thus, by contradiction, let us assume that $\free(f)=1$.

\paragraph{Case 1: \boldmath{$P_{uv}$} is empty.}
By Lemma~\ref{clm:everySuHasNeighbInSv}(b), $S_{uv}= \emptyset$. Since $G$ is non-bipartite, there must be at least one edge inside $S_u$ or $S_v$. Without loss of generality, assume it is in $S_u$. Let $\overrightarrow{D}$ be the graph induced by $S_u$ and oriented with respect to the $f$-orientation; it has at least one nontrivial connected component. If there are more than one nontrivial components, then the desired bound follows from Theorem~\ref{thm:strong}. Hence, we assume that there is exactly one nontrivial component $C$ in $\overrightarrow{D}$. If $C$ has diameter at least~3, then the desired bound follows from Theorem~\ref{thm:strong2}. Thus, we assume that $C$ has diameter at most~2.

Since $C$ is nontrivial, there is at least an arc $\overrightarrow{x_1x_2}$ in $C$. Let $f(x_1x_2)=\overline{x_2y_1}$ for some vertex $y_1 \in N(x_1) \cap S_v$. By Lemma~\ref{clm:everySuHasNeighbInSv}(c), the vertex $x_2$ has also a neighbour $y_2 \in S_v$. Now, the edge $ux_1$ is critical. It cannot be critical for the pair $\{u,x_1\}$ since $u$ and $x_1$ have $x_2$ as a common neighbour, or for a pair $\{u,x\}$ for $x \in S_u$ since $ux \in E$, or for a pair $\{u,y\}$ for $y \in S_v$ since $u$ and $y$ have $v$ as a common neighbour, or for a pair $\{x,y\}$ with $y \in S_v$ because $uy \notin E$. So there is a vertex $x_3 \in S_u$ such that $ux_1$ is critical for the pair $\{x_1,x_3\}$. Furthermore, we have $x_3 \not\in C$ since otherwise $x_1$ and $x_3$ would be at distance at most~$2$, which is not possible by the criticality of $ux_1$. So $x_3$ is independent in $S_u$. Applying the same reasoning, we get a vertex $x_4 \in S_u$, independent in $S_u$, such that $N(x_2) \cap N(x_4) = \{u\}$ (note that we can have $x_3=x_4$).

Since $N(x_1) \cap N(x_3) = \{u\}$ and $x_3$ is independent in $X$, $x_3y_1 \notin E$ and $x_3$ and $y_1$ have a common neighbour $y_3 \in S_v$. But now at least one of the two non-edges $\overline{x_1y_3}$ and $\overline{x_3y_1}$ is $f$-free. Indeed, otherwise we would have two distinct edges $e$ and $e'$ such that $f(e)=\overline{x_1y_3}$ and $f(e')=\overline{x_3y_1}$, which would imply that $e=y_1y_3=e'$, a contradiction. This construction is depicted in Figure~\ref{fig-compDiam2}. Applying the same reasoning, we get a vertex $y_4 \in S_v$ as common neighbour of $x_4$ and $y_2$ (note that we can have $y_3=y_4$), and at least one $f$-free non-edge among $\overline{x_2y_4}$ and $\overline{x_4y_2}$. Since all four non-edges are distinct (this is because $x_1 \neq x_2$, $x_1 \neq x_4$, $x_2 \neq x_3$, $y_1 \neq y_2$, $y_1 \neq y_3$ and $y_2 \neq y_4$), we have two $f$-free non-edges in $G$, a contradiction.

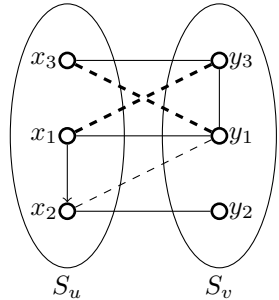
\begin{figure}[ht!]
	\centering
	\begin{tikzpicture}
	\node[jolinoeud] (x1) at (0,1) {};
	\node[jolinoeud] (x2) at (0,0) {};
	\node[jolinoeud] (x3) at (0,2) {};
	\node[jolinoeud] (y1) at (2,1) {};
	\node[jolinoeud] (y2) at (2,0) {};
	\node[jolinoeud] (y3) at (2,2) {};
	\draw[->] (x1)to(x2);
	\draw (y1)to(y3);
	\draw (x1)to(y1);
	\draw (x2)to(y2);
	\draw (x3)to(y3);
	\draw[dashed] (x2)to(y1);
	\draw[dashed,very thick] (x3)to(y1);
	\draw[dashed,very thick] (x1)to(y3);
	\draw (x1) node[left] {$x_1$};
	\draw (x2) node[left] {$x_2$};
	\draw (x3) node[left] {$x_3$};
	\draw (y1) node[right] {$y_1$};
	\draw (y2) node[right] {$y_2$};
	\draw (y3) node[right] {$y_3$};
	\draw (0,1) ellipse (0.75 and 1.75);
	\draw (2,1) ellipse (0.75 and 1.75);
	\draw (0,-1) node {$S_u$};
	\draw (2,-1) node {$S_v$};
	\end{tikzpicture}
	\caption{If the only nontrivial component in $S_u$ has diameter at most~2, one of the two bolded non-edges is $f$-free. Applying the same construction for $x_2$ gives us again one of two non-edges which are $f$-free.}
	\label{fig-compDiam2}
\end{figure}

\paragraph{Case 2: \boldmath{$P_{uv}$} is nonempty.}
 Since $\free(f)=1$, by Lemma~\ref{clm:P_{uv}}(d), $|S_u|\geq |S_v|-1$. As in the proof of Lemma~\ref{clm:P_{uv}}, we let $\free(P_{uv},f)$ be the number of $f$-free non-edges incident with a vertex of $P_{uv}$ (here $\free(P_{uv},f)\in\{0,1\}$). We consider the following subcases.\\

\noindent
\emph{Case 2.1:} Suppose first that the unique $f$-free non-edge is incident with a vertex $p$ of $P_{uv}$, say it is $\overline{pt}$ for some vertex $t\in S_v$. Consider the sets $S_u(p)$ and $S_v(p)$ and the function $m$ as defined in Lemma~\ref{clm:P_{uv}}; by Lemma~\ref{clm:P_{uv}}(c), $S_u(p)$ and $S_v(p)$ are independent sets. By Lemma~\ref{clm:MSBound}, if $||X|-|Y||\geq 2$, then $G$ has at most $\left\lfloor n^2/4\right\rfloor-2$ edges, and we are done. Thus, we assume that $||X|-|Y||\leq 1$. Recall that by Lemma~\ref{eq:Puv}(d), we have $|S_u|\geq |S_v|-1$. Then 
\[
|X| = |S_u|+|S_{uv}|+|P_{uv}|+1 \ge |S_v|+|S_{uv}|+|P_{uv}| = |Y|-1 +|S_{uv}|+|P_{uv}|.
\]
Since $||X|-|Y||\leq 1$, it follows that $|S_{uv}|+|P_{uv}| \in \{1,2\}$, and thus either $|S_{uv}|=|P_{uv}| = 1$ or  $S_{uv} = \emptyset$ and $|P_{uv}| \in \{1,2\}$. We distinguish between these cases.
	\begin{enumerate}
		\item Suppose that $|S_{uv}|=|P_{uv}| = 1$. Then $P_{uv}=\{p\}$ and $|S_u|=|S_v|-1$. The latter implies $S_u=S_u(p)$, and again (as in the case $\free(f)=0$) the vertex in $S_{uv}$ cannot be adjacent to any vertex in $S_u$. Thus, for any non-edge $\overline{xy}$ in $S_u(p)\times S_v(p)$, the only possible common neighbour of $x$ and $y$ is $t$. Thus, if $|S_u(p)|=|S_v(p)|\geq 2$, $t$ must be adjacent to all vertices in $S_u(p)$ and $S_v(p)$. Observe that, if $vt$ is critical for the pair $\{s,t\}$ with $S_{uv} = \{s\}$, then $s$ has no neighbour in $S_v$ (since $t$ is adjacent to every other vertex in $S_v$), and thus all but one of the non-edges between $s$ and $S_v$ are $f$-free (their only possible preimage by $f$ is $sv$). But this is the only pair for which the edge $vt$ can be critical, thus we reach a contradiction.  Furthermore, if $|S_u(p)|=|S_v(p)|=1$, then the vertices in $S_v$ cannot be adjacent to the vertex in $S_{uv}$ since otherwise the edges between $v$ and $S_v$ would not be critical, and the graph has at least two free non-edges ($\overline{pt}$ and one of the non-edges between $S_v$ and $S_{uv}$, which can both only be assigned to the unique edge between $v$ and $S_{uv}$), a contradiction. Since $S_u(p) = S_u \neq \emptyset$, we have finished this case.
		\item Suppose that  $S_{uv} = \emptyset$.  If $S_u=S_u(p)$, then, as in the case $\free(f)=0$, for any $x\in S_u$, the edge $ux$ would not be critical, a contradiction. Thus, we can assume that there is a vertex $w \in S_u \setminus S_u(p)$. Then $P_{uv}=\{p\}$ and $|S_u| = |S_v|$. If $|S_u|=2$, then the vertex in $S_u(p)$ is adjacent to $t$ (since $p$ and $t$ are at distance~2 and $\overline{pt}$ is $f$-free). However, by Lemma~\ref{clm:everySuHasNeighbInSv}(c), $w$ has a neighbour in $S_v$. Thus, $w$ and the vertex in $S_u(p)$ have two common neighbours ($u$ and a vertex in $S_v$), and thus the edge between $u$ and the vertex in $S_u(p)$ is not critical, a contradiction. This implies that there are at least two vertices in $S_u(p)$. Furthermore, $pw \notin E$ (since otherwise this edge would not be critical). Since $\overline{pt}$ is $f$-free, there is a vertex $m(x) \in S_u(p)$ (for a certain $x \in S_v(p)$) such that $m(x)t \in E$ (since otherwise, either $G$ would have diameter~3, or $N(t) \cap N(p) = \{w\}$, two contradictions). Assume first that there exists a vertex $m(y) \in S_u(p)$ (for a certain $y \in S_v(p)$) such that $m(y)t \notin E$. Then the non-edge $\overline{m(y)t}$ must have a preimage by the function $f$, and there are two possibilities. First, if either $f(yt)=\overline{m(y)t}$ or $f(m(x)m(y))=\overline{m(y)t}$, then by Lemma~\ref{clm:P_{uv}}(c) the non-edge $\overline{m(x)y}$ is $f$-free, a contradiction. Now, if $f(m(y)w)=\overline{m(y)t}$ then $wy \in E$ (since otherwise it would be an $f$-free non-edge) and $yt \notin E$. This implies that the non-edge $\overline{m(x)y}$, which is not $f$-free, can only have $m(x)w$ as a preimage by $f$ (since $S_u(p)$ and $S_v(p)$ are independent sets), and thus $wx \in E$ (since otherwise it would be an $f$-free non-edge). However, all this implies that the non-edge $\overline{m(y)x}$ is $f$-free, a contradiction. Thus, $t$ is adjacent to all vertices in $S_u(p)$. However, $t$ is adjacent to no vertex in $S_v(p)$ (since otherwise the edge $vt$ would not be critical), and thus we necessarily have $f(m(x)w)=\overline{m(x)y}$ for all non-edges $\overline{m(x)y} \in S_u(p) \times S_v(p)$, which in turn implies that $w$ is adjacent to all vertices in $S_u(p)$ and $S_v(p)$. Furthermore, $wt \in E$ (since otherwise it would be an $f$-free non-edge: $S_u(p) \subseteq N(w) \cap N(t)$ and $|S_u(p)| \geq 2$). However, $wt$ is not critical, a contradiction.
	\end{enumerate}
	Hence, we have shown that the $f$-free non-edge is not incident with $p$.\\
	
\noindent
\emph{Case 2.2:}	Suppose that the unique $f$-free non-edge is incident with a vertex of $S_{uv}$. Then, $S_{uv}\neq\emptyset$ and $\free(P_{uv},f)=0$. Thus, by Lemma~\ref{eq:Puv}(d), we have $|S_u|\geq |S_v|$. But then, $||X|-|Y||=(|S_u|+|P_{uv}|+|S_{uv}|)-|S_v|\geq 2$, and by Lemma~\ref{clm:MSBound}, we deduce that $G$ has at most $\lfloor n^2/4 \rfloor-2$ edges, a contradiction.\\

\noindent
\emph{Case 2.3:} Finally, assume that the unique $f$-free non-edge is in $S_u\times S_v$. Then, since $\free(P_{uv},f)=0$, by Lemma~\ref{clm:P_{uv}}(d), we deduce that $S_v=S_v(p)$, and $|S_u|\geq |S_v|$. By Lemma~\ref{clm:P_{uv}}(c), $S_v = S_v(p)$ is an independent set. Once again, if $||X|-|Y||\geq 2$, by Lemma~\ref{clm:MSBound}, $G$ has at most $\lfloor n^2/4 \rfloor-2$ edges, and we are done. Thus, there is no vertex in $S_u\setminus S_u(p)$, $P_{uv}=\{p\}$, and $S_{uv}=\emptyset$. Since the $f$-free non-edge is in $S_u \times S_v$, we know that $|S_u(p)|,|S_v(p)|\geq 2$. But now all the edges $ux$ with $x\in S_u(p)$ are not critical, since any two vertices of $S_u$ have both $u$ and $p$ as common neighbours. This is a contradiction.

The above study covers all possible cases, and we always reach a contradiction. Thus, if $P_{uv}\neq\emptyset$, then there are at least two $f$-free non-edges, a contradiction which completes the proof of Theorem~\ref{thm:main}.
\end{proof}

\section{Characterizing D2C graphs of order $n$ with maximum degree $n-2$}\label{sec:highdegree}

It is not difficult to observe that the only D2C graphs of order $n$ with maximum degree $n-1$ (likewise, with minimum degree $1$) are stars. The D2C graphs with maximum degree $n-2$ turn out to be interesting, as they form a precise family of graphs; they all have $2n-4$ edges and a dominating edge. We will see that the graph $H_5$ from Figure~\ref{fig-h5} is the smallest member of this family (indeed it has six vertices and a vertex of degree~$4$).

We first describe a family $\mathcal T$ of twin-free D2C graphs with order $n\geq 6$ and maximum degree~$n-2$. (Recall that twins are non-adjacent vertices with the same neighbours.)
When $n=2k+2$ is an even integer ($k\geq 2$), we let $V(T_n)=A\cup B\cup\{u,v\}$, where $A=\{a_1,\ldots,a_k\}$ and $B=\{b_1\ldots,b_k\}$. The edges of $T_n$ are defined as follows: for every $i$ with $i\leq 1\leq k$, we have the edge $a_ib_i$, the edges $ua_i$, $ub_i$ and $vb_i$. For odd $n$ ($n=2k+3$ and $k\geq 2$), $T_n$ is obtained from $T_{n-1}$ by adding a vertex $w$ adjacent to $u$ and $v$. See Figure~\ref{fig:T-families} for an illustration.

\begin{figure}[!ht]
	\centering
	\begin{tikzpicture}
		\node (Tn) at (-4,0) {
		  \begin{tikzpicture}[scale=0.75]
		    \node[jolinoeud](u) at (6,0.5) {};
                    \draw (u) node[right=0.1cm] {$u$};
		    \node[jolinoeud](v) at (0,-2) {};
                    \draw (v) node[below=0.1cm] {$v$};
		    \node[jolinoeud](a1) at (-2,1) {};
                    \draw (a1) node[left=0.7mm] {$a_1$};
		    \node[jolinoeud](a2) at (-1,1) {};
                    \draw (a2) node[left=0.7mm] {$a_2$};
		    \node[jolinoeud](a3) at (2,1) {};
                    \draw (a3) node[left=0.7mm] {$a_{k-1}$};
		    \node[jolinoeud](a4) at (3,1) {};
                    \draw (a4) node[left=0.7mm] {$a_k$};
		    \node[jolinoeud](b1) at (-2,0) {};
                    \draw (b1) node[left=0.7mm] {$b_1$};
		    \node[jolinoeud](b2) at (-1,0) {};
                    \draw (b2) node[left=0.7mm] {$b_2$};
		    \node[jolinoeud](b3) at (2,0) {};
                    \draw (b3) node[left=0.7mm] {$b_{k-1}$};
		    \node[jolinoeud](b4) at (3,0) {};
                    \draw (b4) node[left=0.7mm] {$b_k$};
		    \node at (0,0) {$\cdots$};
		    \node at (0,1) {$\cdots$};
		    \draw (a1)--(b1);
		    \draw (a2)--(b2);
		    \draw (a3)--(b3);
		    \draw (a4)--(b4);
		    \draw (b3)--(v)--(b4);
		    \draw (b1)--(v)--(b2);
                    \draw[bend right=50] (b1)to(u);
                    \draw[bend right=45] (b2)to(u);
                    \draw[bend right=40] (b3)to(u);
                    \draw[bend right=0,line width=0.75mm] (b4)to(u);
                    \draw[bend right=50] (u)to(a1);
                    \draw[bend right=45] (u)to(a2);
                    \draw[bend right=40] (u)to(a3);
                    \draw[bend right=0] (u)to(a4);

                    \draw[rounded corners] (-2.75, -0.3) rectangle (3.5, 0.3) {};
                    \draw[rounded corners] (-2.75, 0.7) rectangle (3.5, 1.3) {};
		    \node at (-3.25,1) {$A$};
		    \node at (-3.25,0) {$B$};
 		    \node at (1.75,-3) {(a) $T_{2k+2}$};                       
			\end{tikzpicture}
		};
		\node (Tnprime) at (4,0) {
		  \begin{tikzpicture}[scale=0.75]
                    \node[jolinoeud](u) at (6,0.5) {};
                    \draw (u) node[right=0.1cm] {$u$};
		    \node[jolinoeud](v) at (0,-2) {};
                    \draw (v) node[below=0.1cm] {$v$};
		    \node[jolinoeud](w) at (4.5,-2) {};
                    \draw (w) node[below=0.1cm] {$w$};
		    \node[jolinoeud](a1) at (-2,1) {};
                    \draw (a1) node[left=0.7mm] {$a_1$};
		    \node[jolinoeud](a2) at (-1,1) {};
                    \draw (a2) node[left=0.7mm] {$a_2$};
		    \node[jolinoeud](a3) at (2,1) {};
                    \draw (a3) node[left=0.7mm] {$a_{k-1}$};
		    \node[jolinoeud](a4) at (3,1) {};
                    \draw (a4) node[left=0.7mm] {$a_k$};
		    \node[jolinoeud](b1) at (-2,0) {};
                    \draw (b1) node[left=0.7mm] {$b_1$};
		    \node[jolinoeud](b2) at (-1,0) {};
                    \draw (b2) node[left=0.7mm] {$b_2$};
		    \node[jolinoeud](b3) at (2,0) {};
                    \draw (b3) node[left=0.7mm] {$b_{k-1}$};
		    \node[jolinoeud](b4) at (3,0) {};
                    \draw (b4) node[left=0.7mm] {$b_k$};
		    \node at (0,0) {$\cdots$};
		    \node at (0,1) {$\cdots$};
		    \draw (a1)--(b1);
		    \draw (a2)--(b2);
		    \draw (a3)--(b3);
		    \draw (a4)--(b4);
		    \draw (b3)--(v)--(b4);
		    \draw (b1)--(v)--(b2);
		    \draw (v)--(w);
                    \draw[,line width=0.75mm] (w)--(u);
                    \draw[bend right=50] (b1)to(u);
                    \draw[bend right=45] (b2)to(u);
                    \draw[bend right=40] (b3)to(u);
                    \draw[bend right=0,line width=0.75mm] (b4)to(u);
                    \draw[bend right=50] (u)to(a1);
                    \draw[bend right=45] (u)to(a2);
                    \draw[bend right=40] (u)to(a3);
                    \draw[bend right=0] (u)to(a4);

                    \draw[rounded corners] (-2.75, -0.3) rectangle (3.5, 0.3) {};
                    \draw[rounded corners] (-2.75, 0.7) rectangle (3.5, 1.3) {};
		    \node at (-3.25,1) {$A$};
		    \node at (-3.25,0) {$B$};
 		    \node at (1.75,-3) {(b) $T_{2k+3}$};     
			\end{tikzpicture}
		};
	\end{tikzpicture}
	\caption{The two D2C graphs $T_{2k+2}$ and $T_{2k+3}$; $A$ and $B$ are independent sets, and the bold edges are dominating.}
	\label{fig:T-families}
\end{figure}
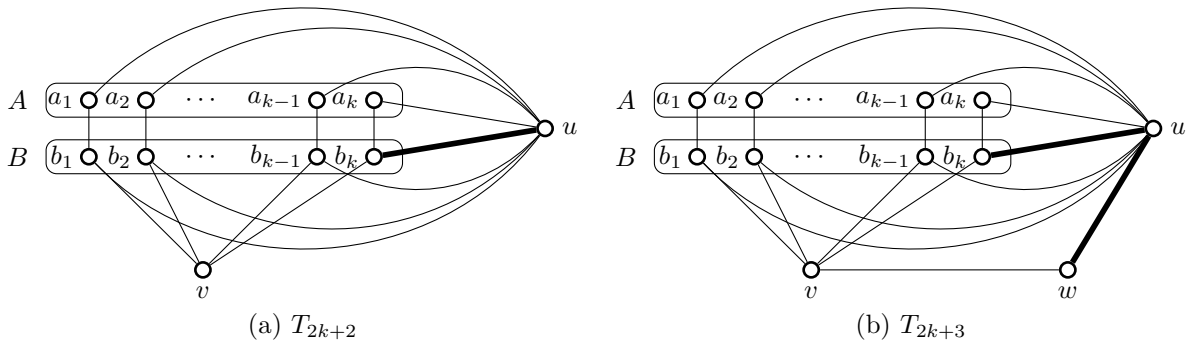

We will extend the family $\mathcal T$ to graphs that have twins. But first, we will need the following theorem of MacDougall and Eggleton~\cite{MacDougall-Eggleton}.

\begin{theorem}[MacDougall and Eggleton {\cite[Theorem 2]{MacDougall-Eggleton}}]\label{thm:twins}
Let $G$ be a D2C graph with a vertex $v$ of $G$. The graph $G'$ obtained from $G$ by adding a twin of $v$ is D2C if and only if for every vertex $w\neq v$ such that $v$ and $w$ are in a common triangle, there exists a non-neighbour $x$ of $v$ in $G$ such that $N(x)\cap N(v)=\{w\}$.
\end{theorem}

We observe the following.

\begin{observation}\label{obs:T-twins}
In the graph $T_n$, the only vertices satisfying the conditions of Theorem~\ref{thm:twins} are the ones of $A\cup\{v,w\}$.
\end{observation}

Nevertheless, adding a twin to $v$ would result in a graph of order $n$ with maximum degree at most $n-3$. Thus, we define the family of graphs $\mathcal T'$ extending $\mathcal T$ in the following way. $\mathcal T'$ contains all graphs obtained from a graph in $\mathcal T$ by replacing an arbitrary (possibly empty) subset of vertices of $A\cup\{w\}$ by any number of twins. Note that for any graph $G$ in $\mathcal T'$, the edge $ub_i$ is dominating for every $i$ with $1\leq i\leq k$. If $w$ exists in $G$, then also the edge $uw$ is dominating.

Next, we show that the graphs in $\mathcal T'$ are D2C. Note that $T_6$ is isomorphic to the graph $H_5$ of Figure~\ref{fig-h5}. The graph obtained from $T_6$ by adding a twin to any vertex of $A$ is isomorphic to the graph of Figure~\ref{fig:thirteen-graphs}(a).\footnote{Similarly, the graph obtained from $T_6$ by adding a twin to $v$ is isomorphic to the graph of Figure~\ref{fig:thirteen-graphs}(b), but it has maximum degree $n-3$ and thus does not belong to $\mathcal T'$.} The graph $T_{7}$ is isomorphic to the graph of Figure~\ref{fig:thirteen-graphs}(c).

\begin{proposition}\label{prop:T-Tprime-D2C}
Any graph of order $n$ in $\mathcal T'$ is D2C and has $2n-4$ edges.
\end{proposition}
\begin{proof}
Rceall that $G$ is obtained from $T_{2k+2}$ or $T_{2k+3}$ ($k\geq 2$) by expanding any (possibly empty) subset of vertices of $A\cup\{w\}$ into twins.
  
It is clear that $T_{2k+2}$ has $4k=2n-4$ edges. When we add $w$ to obtain $T'_{2k+3}$, we add one vertex and two edges. Similarly, when adding a twin of a vertex of $A\cup\{w\}$ to $T_{2k+2}$ or $T_{2k+3}$, we always add one vertex and two edges, thus the resulting number of edges is always $2n-4$.

It remains to show that $G$ is D2C. It is clear that $G$ has diameter~$2$. Assume first that $G$ has no twins. Let $1\leq i,j\leq k$ such that $i\neq j$. The edges $ua_i$ and $ub_i$ are critical for all the pairs $\{a_i,b_j\}$ and $\{b_i,a_j\}$, respectively.
The edge $a_ib_i$ is critical for $\{a_i,v\}$.
Each edge $vb_i$ is critical for the pair $\{v,b_i\}$. The edges $wu$ and $wv$ are critical for the pairs $\{w,a_i\}$ and $\{w,v\}$, respectively.

When $G$ has twins, we apply Observation~\ref{obs:T-twins} and Theorem~\ref{thm:twins}.
\end{proof}

We now show that the graphs of $\mathcal T'$ are the only non-bipartite D2C graphs with maximum degree $n-2$.

\begin{theorem}
Any non-bipartite D2C graph on $n$ vertices and maximum degree~$n-2$ belongs to $\mathcal T'$.
\end{theorem}
\begin{proof}
Let $G$ be a non-bipartite D2C graph on $n$ vertices and maximum degree~$n-2$. By Observation~\ref{obs:T-twins} and Theorem~\ref{thm:twins}, we may assume that $G$ is twin-free; thus we need to prove that $G$ belongs to $\mathcal T$.

Let $u$ be a vertex of degree $n-2$ in $G$, and let $v$ be its unique non-neighbour. First of all, we claim that $N(v)$ forms an independent set. Indeed, if $x,y$ are two adjacent vertices of $N(v)$, the edge $xy$ cannot be critical, a contradiction.

Next, we claim that every vertex $x$ of $N(u)\setminus N(v)$ has exactly one neighbour in $N(v)$. First, if $x$ has no neighbour in $N(v)$, then $x$ and $v$ would be at distance~$3$, a contradiction. Second, if $x$ had two distinct neighbours $y$ and $z$ in $N(v)$, then the edges $xy$ and $xz$ could not be critical, a contradiction. Denote by $f(x)$ the unique neighbour of $x$ in $N(v)$.

We now show that $N(u)\setminus N(v)$ is an independent set. Indeed, if we had two adjacent vertices $x_1$ and $x_2$ in $N(u)\setminus N(v)$, since both $x_1$ and $x_2$ have a neighbour in $N(v)$, the edge $x_1x_2$ cannot be critical, a contradiction.

We now show that if $x_1,x_2$ are two distinct vertices of $N(u)\setminus N(v)$, then $f(x_1)\neq f(x_2)$. Indeed, if we had $f(x_1)=f(x_2)$, then $x_1$ and $x_2$ would be twins, a contradiction.

Thus, the subgraph induced by $N(u)$ is a collection of disjoint edges. Hence, there can be at most one vertex in $N(v)$ that has no neighbour in $N(u)$ (if there are two they would be twins).

Thus, we let $A=N(u)\setminus N(v)$, $B=\{y\in N(v)~|~y=f(x)\text{ and $x\in A$}\}$, and if $(B\setminus N(v))\neq\emptyset$, we let $w$ be the only vertex of $B\setminus N(v)$. It is now clear that $G$ is isomorphic to $T_n$.
\end{proof}

\section{Conclusion}\label{sec:conclu}

Conjecture~\ref{conj} postulates that there is a linear gap in the set of possible numbers of edges of D2C graphs of order $n$ when we exclude the well-understood class  of complete bipartite graphs: from $\lfloor n^2/4\rfloor$ to $\lfloor (n-1)^2/4 \rfloor+1$. The bound of Theorem~\ref{thm:main} only shows a constant gap for graphs with a dominating edge, which leaves room for further improvements. We hope that our method can be further used to improve the gap by a function of $n$ (ideally linear), towards Conjecture~\ref{conj}. As witnessed by Theorems~\ref{thm:strong} and~\ref{thm:strong2}, in order to do so, one should first focus on the case when $P_{uv}$ is nonempty. When $P_{uv}$ is empty, the first case to improve is the one considered in Theorem~\ref{thm:strong2}, especially when there is a unique connected component in both $S_u$ and $S_v$, and the $f$-orientation is acyclic.

Recall that the infinite family $\F$ of extremal graphs for Conjecture~\ref{conj} contains only graphs with no dominating edge. Moreover, among graphs of order at most $11$, there are only ten non-bipartite D2C graphs with a dominating edge (all of order at most $9$) and $\lfloor(n-1)^2/4\rfloor+1$ edges. Thus we suspect that, for this class of D2C graphs, the bound of Conjecture~\ref{conj} is actually not tight, and ask the following.

\begin{question}
What is the largest possible number of edges of a non-bipartite D2C graph of order $n$ and with a dominating edge?
\end{question}

Towards this question, one possibility to obtain a D2C graph with a dominating edge is to start with the graph $T_7$ defined in Section~\ref{sec:highdegree} and depicted in Figure~\ref{fig:thirteen-graphs}(c), and expand the vertices $v$ and $w$ into two equal-size sets of twins. The resulting graph is D2C, has $\tfrac{(n-2)^2+15}{4}$ edges, and the edge $uw$ is dominating.




\end{document}